\documentclass[addpoints]{amsart} 
\usepackage[hidelinks]{hyperref} 
\hypersetup{allcolors=black}

\usepackage{gensymb} 
\usepackage{pgf, tikz} 
\usepackage{tikz-cd} 
\usetikzlibrary{angles,arrows.meta,automata,backgrounds,calc,decorations.markings,decorations.pathreplacing,intersections,patterns,positioning,quotes} 
\usetikzlibrary{shapes} 
\usepackage{tkz-euclide} 
\usepackage{etoolbox}
\usepackage{placeins}

\usepackage{mdframed} 
\usepackage{pgfplots}
\usepackage{tikz-3dplot} 
\pgfplotsset{width=10cm,compat=1.9}
\usepackage[export]{adjustbox}
\usepackage{caption}
\usepackage{subcaption}
\usepackage{wrapfig}
\usepackage{graphicx}
\usepackage{mathtools}  
\usepackage{amsthm,thmtools} 
\usepackage{relsize} 
\usepackage{array}   
\newcolumntype{L}{>{$}l<{$}} 

\usepackage{amsmath,systeme} 

\usepackage{fancyhdr} 

\usepackage{fdsymbol} 
\usepackage{dutchcal} 
\usepackage{courier}
\usepackage[scr=rsfso]{mathalfa} 
\input ArtNouvc.fd
\input Starburst.fd

\usepackage[all,cmtip]{xy}

\usepackage{enumerate}
\usepackage{multicol} 
\usepackage{mdwlist}

\usepackage[OT2,T1]{fontenc}
\DeclareSymbolFont{cyrletters}{OT2}{wncyr}{m}{n}
\DeclareMathSymbol{\Sha}{\mathalpha}{cyrletters}{"58}

\theoremstyle{plain}
\newtheorem{Theorem}[subsubsection]{Theorem}
\newtheorem{Lemma}[subsubsection]{Lemma}

\newtheorem{Proposition}[subsubsection]{Proposition}

\newtheorem{Definition}[subsubsection]{Definition}

\theoremstyle{remark}

\newtheorem{Examples}[subsubsection]{\bf Examples}

\newtheorem{Remark}[subsubsection]{\bf Remark}

\numberwithin{equation}{subsubsection}

\newcommand{\bb}[1]{{\bold{#1}}}
\newcommand{\bs}[1]{{\boldsymbol{#1}}}

\newcommand{\mbb}[1]{{\mathbb{#1}}}

\newcommand{\mf}[1]{{\mathfrak{#1}}}

\renewcommand{\a}{{\bs a}}
\renewcommand{\b}{{\bs b}}
\renewcommand{\c}{{\bs c}}

\newcommand\bsa{{\bs a}}
\newcommand\bsb{{\bs b}}
\newcommand\bsc{{\bs c}}

\newcommand{\h}{{\text{\rm h}}}
\renewcommand{\i}{{\boldsymbol i}}
\renewcommand{\j}{{\boldsymbol j}}

\renewcommand\u{{\boldsymbol u}}
\renewcommand\v{{\boldsymbol v}}
\newcommand\w{{\boldsymbol w}}
\newcommand\x{{\bold x}}

\newcommand\0{{\boldsymbol 0}}
\renewcommand\1{{\boldsymbol 1}}

\newcommand{\A}{{\mathbb A}}

\newcommand{\C}{{\mathbb C}}

\newcommand{\DD}{{\mathfrak D}}

\renewcommand{\P}{{\mathbb P}}

\newcommand{\R}{{\mathbb R}}

\renewcommand{\S}{{\mathsf{S}}}
\newcommand{\T}{{\mathsf{T}}}

\newcommand{\TT}{{\mathfrak{T}}}

\newcommand{\V}{{\mathsf V}}

\newcommand{\X}{{\text{\rm X}}}

\newcommand{\Z}{{\mathbb Z}}

\newcommand{\sdbmx}[1]{{\text{\footnotesize $\begin{bmatrix}#1\end{bmatrix}$}}} 

\newcommand{\sdvmx}[1]{{\text{\footnotesize $\begin{vmatrix*}#1\end{vmatrix*}$}}} 
\newcommand{\br}[1]{{\langle{#1}\rangle}}

\newcommand{\height}{{\text{\rm ht}\,}}
\newcommand{\id}{{\text{\rm id}}}

\newcommand{\isim}{{\,\begin{tikzpicture}\draw[->,-latex](0,0)--(.75,0);\node at (.35,.1) {$\sim$};\end{tikzpicture}\;}}
\newcommand{\isom}{{\;\simeq\;}}

\newcommand{\lm}{\longmapsto}
\newcommand{\lr}{\longrightarrow}
\newcommand\dashlr{\,\tikz[baseline]{\draw[dashed, ->, -latex] (0,.1)--(.55,.1);}\;}

\newcommand{\mywedge}[1]{{\bigwedge}^{\kern-1pt{#1}}}

\newcommand{\ol}[1]{{\overline{#1}}}

\renewcommand{\pmod}[1]{{\text{\rm (mod ${#1}$)}}}

\newcommand{\rk}{{\text{\rm rk}}}

\newcommand{\sfrac}[2]{{\frac{\text{\footnotesize$#1$}}{\text{\footnotesize$#2$}}}}

\newcommand{\widesim}[1][1.5]{\mathrel{{\scalebox{#1}[1]{$\sim$}}}} 
\newcommand\wt[1]{{\widetilde{#1}}}

\newcommand{\Arg}{{\sf{Arg}}}

\newcommand{\Bl}{{\text{\rm Bl}}}

\renewcommand{\Im}{{\text{\rm Im}}}

\newcommand{\Sim}{{\mathsf{Sim}}}

\newcommand{\SU}{{\mathsf{SU}}}

\newcommand{\Var}[1]{\mathsf{#1}\text{-}\mathsf{Var}}

\renewcommand{\S}{{\sf{S}}}
\renewcommand{\V}{{\rm{V}}}

\renewcommand{\a}{{\bsa}}
\renewcommand{\b}{{\bsb}}
\renewcommand{\c}{{\bsc}}
\renewcommand{\1}{{\bf 1}}

\newcommand\dbl{{\text{\rm dbl}}}
\renewcommand\X{{\mbb{X}}}
\renewcommand\S{{\mbb{S}}}
\renewcommand\T{{\mbb{T}}}
\renewcommand\TT{{\mathfrak T}}
\renewcommand\DD{{\mathfrak D}}

\definecolor{maroon}{RGB}{128, 0, 0}
\definecolor{darkgreen}{RGB}{0, 100, 0}
\definecolor{schoolbusyellow}{RGB}{255, 216, 0}
\definecolor{burntorangeclassic}{RGB}{204, 85, 0}
\definecolor{burntorange}{RGB}{191, 87, 0}

\newcommand{\hlr}{{\raisebox{1pt}{\,\begin{tikzpicture}\draw[right hook-stealth, shorten <= 3pt, shorten >= 3pt] (0,0)--(1,0);\end{tikzpicture}\;}}}
\newcommand{\llr}{{\,\begin{tikzpicture}\draw[->,-latex](0,0)--(.75,0);\node at (.35,-.01) {$\hphantom{\sim}$};\end{tikzpicture}\;}}
\newcommand{\FM}[1]{{\A^2(\R)[#1]}}
\newcommand\dgen{{\text{\rm dgen}}}
\newcommand\ndgen{{\text{\rm ndgen}}}
\newcommand\coll{{\text{\rm coll}}}
\newcommand\tpl{{\text{\rm tpl}}}

\author[Brussel, Goertz, Guptill, Lyle]{Eric Brussel, Madeleine Goertz, Elijah Guptill, Kelly Lyle}
\thanks{The second author thanks the Frost Foundation for their generous support in the form of
the Frost Research Scholarship. The third and fourth authors thank the Frost Foundation for
their support during the 2023 Frost Summer Research Program.}

\begin{document}

\small
\title[A Smooth Compactification of Triangle Classes]
{A Smooth Compactification of Similarity Classes of Triangles}

\begin{abstract}
We provide a smooth algebraic compactification 
of the space of labeled, oriented similarity classes of triangles,
as a quotient of a 3-point configuration space in the plane
-- minus triple-collisions -- by the similarity group.
The configuration space is naturally isomorphic to the Fulton-MacPherson space of \cite{FM94}
and to the Schubert-Semple space of \cite{Semple}.
The quotient is topologically a connected sum of three projective planes, 
and projects via blow-down to two shape spaces that have appeared in the literature:
the well-known 2-sphere (\cite{Kend84}, \cite{Beh}, \cite{Montgomery}, \cite{ES15}), 
and the less-well-known 2-torus (\cite{BG23}).
A natural action by the dihedral group $D_6$ yields a quotient stack 
describing absolute (unlabeled, unoriented) triangle classes.
\end{abstract}


\maketitle

\tableofcontents

\newpage

\section{\Large Introduction}

The study of triangles in the plane dates back to Euclid, and 
is one of the oldest and most thoroughly investigated
subjects in all of mathematics. 
The literature on 
triangle constructions, special points, connections with Apollonian problems, 
and the Euclidean geometry of the plane is rich and enormous.
For a sense of its depth one has only to browse the list of geometrically 
defined triangle ``centers'' in 
\cite{Kim25}. There are thousands.

In 1880 Schubert \cite{Schubert} constructed what is now viewed as a smooth
compactification of the space of labeled, oriented, nondegenerate triangles,
as part of his enumerative calculus.
Schubert studied enumerative problems involving the number of triangles in a suitably
defined intersection of triangle families (\cite[p.80]{Semple}),
like the number of triangles simultaneously inscribed in one plane curve and circumscribed by another
(\cite[p.92]{ColFul89}).
His construction was made more rigorous by Study \cite{Study01}
and by Semple \cite{Semple}, and excellent summaries and improvements of the final product appear in 
\cite{RSp81}, \cite{ColFul89}, and \cite{vdKM99}.

The space of triangles in the plane is an introductory example of a
configuration space of $n$ distinct points on a scheme or algebraic manifold,
the construction of which is an important problem with an extensive literature, 
see e.g. \cite{FM94}, \cite{AxSi94}, and \cite{Sinha}.
The basic problem of compactifying such a space over a scheme
appears already with two points in $\P^2(\R)$. 
The exclusion of the diagonal $\Delta$ of $\P^2(\R)\times\P^2(\R)$, 
where the points are not distinct, leaves a
non-compact space whose geometry is singular at $\Delta$, 
in the sense that different approach directions to a point
become identified in the limit.
Simply including $\Delta$ compactifies the space but does not capture this
scheme-theoretic aspect of the geometry of degenerating families.
A different compactification resolves the problem, in this case the blow-up of
$\P^2(\R)\times\P^2(\R)$ along $\Delta$.

Parameter spaces of triangles {\it up to similarity}, called {\it shape spaces} in \cite{Kend84}, 
are quotients of triangle spaces, or
``what is left when we quotient out location,
size, and rotation effects'' \cite[p.467]{Kend86}.
The problem of constructing such spaces is old,
with examples dating back to the nineteenth century in \cite{WSB} and \cite{Dod58}. More recent papers 
include \cite{Kend84}, \cite{Port}, 
\cite{Beh}, \cite{Montgomery}, \cite{ES15}, \cite{Anderson}, \cite{CNSS},  and \cite{BG23}.
In particular, the book \cite{Kend99} by
Kendall, Barden, Carne, and Le gives a well-written and comprehensive account of shape theory,
including the shape spaces
$\Sigma_m^k$ of $k$-points in $\R^m$, along with their
homology and Riemannian structure.
An important application is in the probability and statistics of shapes 
that can be parameterized by finitely many points in $\R^m$.
Thus their stated goal is to 
``define shape in a manner that is susceptible to mathematical and statistical analysis'' 
\cite[Preface]{Kend99}.
Shape theory in its modern form thus connects directly to the original questions posed in 
\cite{WSB} and \cite{Dod58}
about the ratio of acute to obtuse triangles.  
Kendall's shape space $\Sigma_2^3$ is the two-sphere of radius $1/2$,
with the great-circle metric, constructed as a quotient
of $(\R^2)^3$ by the similarity group (\cite{Kend85}). 
The sphere construction separates collinear classes by side-length ratio and locates them
on the sphere as a great circle, but collapses all double-vertex degenerations to three points,
corresponding to the three possible zero-length sides.

In this paper we construct a smooth compact shape space of labeled and oriented triangles
in the plane as a quotient $\DD:=\TT^\circ/\Sim^+(2)$
of our $3$-point configuration space $\TT$ minus the triple-collision divisor.
Here $\Sim^+(2)$ is the group of orientation-preserving similarities of the plane, 
written in complex notation as $\C\rtimes\C^\times$.
The space $\TT$ is naturally isomorphic to the Fulton-MacPherson space $\FM{3}$, and to the Schubert-Semple space $W_6^*$ in \cite[Theorem II]{Semple}.
Our quotient $\DD$ separates collinear classes by side-length ratio,
and separates double-vertex classes by approach directions.
It thus reconciles two different shape spaces 
in the literature, the $2$-sphere mentioned above and the $2$-torus of \cite{BG23}.
The space $\DD$ is isomorphic to a connected sum of three projective planes, denoted $\#_3\P^2(\R)$, and in this form is known as {\it Dyck's surface}.
It is named after Walther van Dyck, who proved in 1888 (\cite{Dyck88})
that it is isomorphic to the connected sum of a torus and a projective plane, $\mbb T\#\P^2(\R)$,
thereby providing what turns out to be the only relation in the (non-unique) factorization into primes of a
compact surface in $\R^3$. 
Once we construct $\DD$ we form another quotient $\DD/D_6$ by a natural $D_6$-action,
and obtain a stack of unrestricted similarity classes.

Our construction describes triangles and triangle classes 
by specific continuous {\it geometric variables} (Theorem~\ref{main1} and Lemma~\ref{orbits}).
The variable for triangles is similar to the one used in \cite{Semple} and \cite{ColFul89}, and parameterizes
the explicit configuration space of \cite{FM94}.
The variables are the natural coordinatizations of the blow-ups 
that realize the spaces, and
thus while they are redundant on nondegenerates,
they record the infinitesimal data we seek at the degenerations.

Different characterizations and notions of degeneracy
in general produce different compactifications, hence different shapes.
There are other conceivable criteria, 
and we do not claim to have accounted for all of them. 
Our construction is designed to align with both Schubert-Semple's and Fulton-MacPherson's 
classical triangle spaces, while also resolving the sphere-torus conflict:
It preserves approach directions of degenerating families at
double points, while also tracking collinear classes by side-length ratio.

\subsection{Missing Each Other's Degenerate Classes}\label{mdc}

To illustrate the difference between how the torus $\mbb T$ and the sphere $\S^2$
model the geometry of degenerating families, consider
the following two examples.

\begin{Examples}
\eqref{constantangle} shows two families with different constant interior angles $\alpha$ at $A$.
They are represented on $\mbb T$ as the distinct lines $\alpha=\pi/2$ and $\alpha=2\pi/3$, 
and define distinct degenerate limit points on $\mbb T$. But on $\S^2$ they converge to a common limit point.

\begin{equation}\label{constantangle}
\begin{minipage}{.4\textwidth}
\centering
\begin{tikzpicture}[scale=1.1]
\def\radius{1.3}
\draw[line width=0.2mm, domain=0:180, smooth, variable=\t, dashed] plot ({\radius*cos(\t)},{\radius*sin(\t)}); 
\coordinate (target) at ({\radius*cos(180)},{\radius*sin(180)}); 
\coordinate (target2) at ({\radius*cos(0)},{\radius*sin(0)}); 
\foreach \angle in {0,10,25,45,90}{
\coordinate (vertex) at ({\radius*cos(\angle)},{\radius*sin(\angle)});
\fill (vertex) circle (1pt);
\fill (-\radius,0) circle (1pt);
\draw (vertex) -- (target);
\draw (vertex) -- (target2);
\draw[thick] (target)--(target2);
}
\draw[->,-latex] (0,{\radius})--(.4,{\radius});
\fill (target) circle (1.5pt);
\fill (target2) circle (1.5pt);
\draw (target2) circle (2.5pt);
\node[below right] at (target2) {\footnotesize $C$};
\node[below left] at (target) {\footnotesize $B$};
\node[above] at ({\radius*cos(90)},{\radius*sin(90)}) {\footnotesize $A$};
\node at (0,-.65) {\footnotesize Family of Right Classes};
\node at (0,-1) {\footnotesize Converging to a Double Point};
\end{tikzpicture}
\end{minipage}
\begin{minipage}{.4\textwidth}
\begin{tikzpicture}[scale=1.1]
\def\radius{2}
\draw[line width=0.2mm, domain={acos(1.3/2)}:{180-acos(1.3/2)}, smooth, variable=\t, dashed] plot 
({\radius*cos(\t)},{-sqrt(4-(1.3)^2)+\radius*sin(\t)}); 
\coordinate (target) at ({\radius*(-1.3/2)},{-sqrt(4-(1.3)^2)+\radius*sin(acos(1.3/2))}); 
\coordinate (target2) at ({\radius*(1.3/2)},{-sqrt(4-(1.3)^2)+\radius*sin(acos(1.3/2))}); 
\foreach \angle in {acos(1.3/2),65,80,95,180-acos(1.3/2)}{
\coordinate (vertex) at ({\radius*cos(\angle)},{-sqrt(4-(1.3)^2)+\radius*sin(\angle)});
\fill (vertex) circle (1pt);
\draw (vertex) -- (target);
\draw (vertex) -- (target2);
\draw[thick] (target)--(target2);
}
\draw[->,-latex] ({\radius*cos(95)},{-sqrt(4-(1.3)^2)+\radius*sin(95)})--
({\radius*cos(95)+.4*sin(95)},{-sqrt(4-(1.3)^2)+\radius*sin(95)-.4*cos(95)});
\fill (target) circle (1.5pt);
\fill (target2) circle (1.5pt);
\draw (target2) circle (2.5pt);
\node[below right] at (target2) {\footnotesize $C$};
\node[above left] at (target) {\footnotesize $B$};
\node[above] at ({\radius*cos(95)},{-sqrt(4-(1.3)^2)+\radius*sin(95)}) {\footnotesize $A$};
\node at (0,1.6) {}; 
\node at (0,-.55) {\footnotesize Family of Obtuse Classes};
\node at (0,-.9) {\footnotesize Converging to a Double Point};

\end{tikzpicture}
\end{minipage}
\end{equation}
On the other hand,
\eqref{constantsideratio} shows two families with distinct side-length ratios, one isosceles $|\c|=|\b|$,
and the other with $|\c|=2|\b|$. On $\mbb T$ they
trace a line and a curve, respectively,
converging to the degenerate point $(\alpha,\beta,\gamma)=(\pi,0,0)$.
But on $\S^2$ they have distinct limit points, separated by their distinct side-length ratios.
\begin{equation}\label{constantsideratio}
\begin{minipage}{.4\textwidth}
\centering
\begin{tikzpicture}[scale=1.1]
\def\radius{1.3}
\coordinate (target) at ({\radius*cos(180)},{\radius*sin(180)}); 
\coordinate (target2) at ({\radius*cos(0)},{\radius*sin(0)}); 
\foreach \height in {0,.15,.4,.7,1.3}{
\coordinate (vertex) at (0,{\height});
\fill (vertex) circle (1pt);
\fill (-\radius,0) circle (1pt);
\draw (vertex) -- (target);
\draw (vertex) -- (target2);
\draw[dashed] (0,0)--(0,1.3);
\draw[thick] (target)--(target2);
}
\draw[->,-latex] (0,{\radius*sin(90)})--
(0,{\radius*sin(90)-.4});
\fill (target) circle (1.5pt);
\fill (target2) circle (1.5pt);
\fill (0,0) circle (1.5pt);
\node[below right] at (target2) {\footnotesize $C$};
\node[below left] at (target) {\footnotesize $B$};
\node[above] at ({\radius*cos(90)},{\radius*sin(90)}) {\footnotesize $A$};
\node at (0,-.65) {\footnotesize Family of Isosceles Classes};
\node at (0,-1) {\footnotesize Converging to a Collinear Point};
\end{tikzpicture}
\end{minipage}
\begin{minipage}{.4\textwidth}
\begin{tikzpicture}[scale=1.1]
\def\radius{1.3}
\draw[line width=0.2mm, domain=0:180, smooth, variable=\t, dashed] plot ({\radius*cos(\t)},{\radius*sin(\t)}); 
\coordinate (target) at ({-2*\radius},0); 
\coordinate (target2) at ({-\radius/2},0); 
\foreach \angle in {45,110,140,160,170}{
\coordinate (vertex) at ({\radius*cos(\angle)},{\radius*sin(\angle)});
\fill (vertex) circle (1pt);
\fill ({\radius},0) circle (.5pt);
\draw (vertex) -- (target);
\draw (vertex) -- (target2);
\draw[thick] (target)--(target2);
}
\draw[->,-latex] ({\radius*cos(45)},{\radius*sin(45)})--
({\radius*cos(45)-.2*sqrt(2)},{\radius*sin(45)+.2*sqrt(2)});
\fill (target) circle (1.5pt);
\fill (target2) circle (1.5pt);
\fill ({-\radius},0) circle (1.5pt);
\node[below right] at (target2) {\footnotesize $C$};
\node[below left] at (target) {\footnotesize $B$};
\node[above right] at ({\radius*cos(45)},{\radius*sin(45)}) {\footnotesize $A$};
\node at (-.65,-.65) {\footnotesize Family of Scalene $(2:1)$ Side-Ratio Classes};
\node at (-.65,-1) {\footnotesize Converging to a Collinear Point};
\end{tikzpicture}
\end{minipage}
\end{equation}
In each example, distinct limit points on one model are identified on the other.
This is a general phenomenon with these models: On $\mbb T$ the degenerate classes 
consist entirely of double points, exactly one of which is collinear, 
and on $\mbb S^2$ they consist of collinear points, exactly three of
which are double points. The nondegenerate interiors,
on the other hand, are diffeomorphic, 
and there is an easily constructed rational isomorphism $\mbb T\dashlr\mbb S^2$.
\end{Examples}

\subsection{Literature}
In the seminal paper \cite{FM94},
Fulton and MacPherson construct a configuration space $X[n]$ for $n$ points on a nonsingular
scheme $X$. They introduce the concept of 
{\it screens} to parameterize degenerations,
and then perform a delicate sequence of blow-ups to produce a smooth scheme
whose degenerate locus, which classifies the degenerate types, 
is a normal crossings divisor. 
This paper is the basis for an entire literature, 
mostly focused on the complexities that arise when $n\geq 4$.
We use our version of their $n=3$ case $\FM{3}$ as the parent space for our quotient.
Fulton-MacPherson's compactification is a fine moduli space for a functor of screens
(\cite[Theorem 4]{FM94} and \cite[{\it A space of $n$-pointed stable degenerations}, p.194]{FM94}), 
and their space $\P^1(\C)[4]$ is isomorphic to Kontsevich's
(coarse) moduli space $\ol M_{0,4}(\P^1,1)$ of $4$-pointed stable maps of degree one (\cite[Remark 1.5]{Mass17}).
Different blow-up schemes and generalizations of 
Fulton-MacPherson's construction are discussed in \cite{Uly02}, \cite{Li09}, \cite{Mass17}, and \cite{GGR22}.

The torus, as a shape space of
triangle {\it classes}, appears to our knowledge only in \cite{BG23}.
It is a smooth compactification of
a space defined by triples of angles, as opposed to vertices-modulo-similarity.
The torus has a natural group structure whose subgroups express different triangle
types, and a (flat) metric that allows explicit computations.

The sphere is more commonly cited as a shape space of triangles,
and is the only other one we have found in the literature. 
An excellent in-depth study with connections to random triangle theory appears in \cite{ES15}.
The sphere as a space of triangles is a fundamental example in {\it shape theory}
(\cite{Kend99}), as mentioned above, which begins with a series of papers by 
Kendall (see \cite{Kend84}, \cite{Kend85}, \cite{Kend86}, and \cite{Kend89}).
Their goal is to 
construct the shape space $\Sigma_m^k$ of $k$ points in $\R^m$,
study its geometry, and make measurements in service of probabilistic and statistical analyses, for example 
in archaeology and astronomy.
Kendall derives $\Sigma_2^3$, which is a $2$-sphere of radius $1/2$, 
as a quotient (\cite[Theorem 2]{Kend84}). We also construct the 2-sphere below, using a similar method,
and arriving at a similar interpretation of the different regions 
(\cite[Figure 3]{Kend89}, \cite[p.16, Figure 1-1]{Kend99}). 
The geometry
is of primary importance here, and is dictated by the applications for which shape theory was developed. 

Behrend \cite{Beh} conducts a thorough investigation of triangle classes by way
of introducing algebraic stacks, with many examples
of fine moduli spaces and stacks determined by different restrictions on classes.
In particular, he notes that ``there are many ways to think of degenerate
triangles, all giving rise to different compactifications'' (\cite[p.44]{Beh}).
He uses the (collinear) notion of degenerate 
to construct a smooth compactification of the fine moduli space of 
nondegenerate classes, isomorphic to the 2-sphere, together with a natural $S_3\times\Z_2$-action
(\cite[Exercises 1.37, 1.38, Figure 1.24]{Beh}).

\subsection{Smooth Manifolds and $\R$-Varieties}\label{thesame}
The torus and sphere in $\R^3$, and all of the spaces we consider in this paper,
are {\it algebraic manifolds}, which are
smooth {\it $\R$-varieties} that are naturally smooth manifolds.
They include smooth affine and projective varieties over $\R$ and their blow-ups along smoothly embedded divisors.
Technically, for us an $\R$-variety is an integral, separated scheme of finite type over $\R$.
We let $\Var{\R}$ denote the category of smooth $\R$-varieties. 

Any complex variety $X$ has an underlying real structure, which we denote by $X_\R$ to avoid confusion. In particular, with $\X\isom\A^2(\C)$ we 
denote by $\P(\X)_\R$ the underlying $\R$-variety of $\P(\X)\isom\P^1(\C)$, and
$\P(\X_\R)$ the real projective space, isomorphic to $\P^5(\R)$.
Since $\A^1(\C)$ and $\A^2(\R)$ are equal as point sets, 
configurations of $3$ points in the plane can refer to either $\A^1(\C)[3]$ or $\A^2(\R)[3]$, but
we will use the latter.
We will use the complex notation when convenient, for example
when describing similarities in the plane or parameterizing side-vectors.

\subsection{Prerequisites}
We require algebraic geometry on the level of \cite{Ha92} or \cite{Liu},
primarily for the construction of blow-ups of $\R$-varieties.
In Section~\ref{mss} we discuss how the quotient of Dyck's surface
$\DD$ by the action of a finite group
forms a quotient stack, which in this case is a slight generalization of a variety. 
We give the definitions necessary to interpret this statement. 

\subsection{Outline}
The outline of the paper is as follows.
\begin{enumerate}
\item
[Section~\ref{pf}:]
We discuss the families of triangle classes given by {\it Poncelet's Porism}.
\item
[Section~\ref{dn}:] We define triangles with labels, orientation,
and degeneracy, and construct the space $\TT$.
\item
[Section~\ref{triangleclasses}:]
We define similarity, and prove the quotient $\DD=\TT^\circ/\Sim^+(2)$ is a smooth $\R$-variety
isomorphic to Dyck's surface, where $\TT^\circ$ deletes the (redundant) triple-collision divisor of $\TT$.
\item
[Section~\ref{discussion}:]
We prove $\TT\isom\FM{3}$, give a point-theoretic correspondence,
discuss the parameterization of double points, and justify the inclusion of orientation.
\item
[Section~\ref{ssss}:] We prove that two blow-downs of Dyck's surface, 
to sphere and torus, recover the above-mentioned
shape space constructions of triangle classes in the literature.
\item
[Section~\ref{mss}:]
We quotient Dyck's surface by a natural action by $D_6$ to form a (smooth) stack
of unrestricted classes.
\end{enumerate}

\subsection{Acknowledgement}
The authors thank Quinn Brussel for helpful comments and suggestions. 

\section{\Large Families of Classes}\label{pf}
\subsection{Poncelet Families}
Any triangle is uniquely inscribed in
and circumscribed by two circles: 
\[\begin{tikzpicture}[scale=.5] 
\draw (1,0) circle (2);
\draw (0,0) circle (3/4);
\filldraw (1.744, 1.857) circle (1pt); 
\filldraw (-.976, .307) circle (1pt); 
\filldraw (-.01, -1.726) circle (1pt); 
\draw (1.744, 1.857)--(-.976, .307)--(-.01, -1.726)--(1.744, 1.857);
\end{tikzpicture}\]
Chapple proved in 1746 (\cite[Theorem 295]{J60})
that conversely two circles of radius $r$ and $R$
are incircle and outcircle for a triangle
if and only if their centers are separated by a distance $d>0$
uniquely determined by the quadratic relation $(R-r)^2=r^2+d^2$. 
Since $r$ and $R$ determine $d$, the configurations
of circles in the plane
for which a ``Poncelet triangle'' exists have measure zero. 

On the other hand, if a configuration admits one Poncelet triangle,
{\it Poncelet's Porism} \cite[Theorem 297]{J60} states that it admits infinitely many,
forming a {\it continuous family}. The family is
obtained by revolving the one triangle around the perimeter of the outcircle:
\[\begin{tikzpicture}[scale=.75]
\draw (0,0) circle (3/4);
\draw (1,0) circle (2);
\foreach \x in {-90, -60,-30,-15,0,15,30,60,90} {
\coordinate (A) at ({(cos(\x)+sqrt(4-(sin(\x))^2))*cos(\x)},{(cos(\x)+sqrt(4-(sin(\x))^2))*sin(\x)});
\coordinate (B) at ({((cos(\x)+sqrt(4-(sin(\x))^2))*cos(\x))
-((-2/3)*(cos(\x)-sqrt(4-(sin(\x))^2))*((55/16+cos(2*\x)+2*cos(\x)*sqrt(4-(sin(\x))^2))^.5*(sqrt(4-(sin(\x))^2))
-(3/4)*sin(\x))
*((1/4)*((12-cos(2*\x)+2*cos(\x)*sqrt(4-(sin(\x))^2))^.5*cos(\x)-(cos(\x)-sqrt(4-(sin(\x))^2))*sin(\x)))},
{(cos(\x)+sqrt(4-(sin(\x))^2))*sin(\x)
-((-2/3)*(cos(\x)-sqrt(4-(sin(\x))^2))*((55/16+cos(2*\x)+2*cos(\x)*sqrt(4-(sin(\x))^2))^.5*(sqrt(4-(sin(\x))^2))
-(3/4)*sin(\x))
*((1/4)*((12-cos(2*\x)+2*cos(\x)*sqrt(4-(sin(\x))^2))^.5*sin(\x)+(cos(\x)-sqrt(4-(sin(\x))^2))*cos(\x)))});
\coordinate (C) at ({(cos(\x)+sqrt(4-(sin(\x))^2))*cos(\x)
-((-2/3)*(cos(\x)-sqrt(4-(sin(\x))^2))*((55/16+cos(2*\x)+2*cos(\x)*sqrt(4-(sin(\x))^2))^.5*(sqrt(4-(sin(\x))^2))
+(3/4)*sin(\x))
*((1/4)*((12-cos(2*\x)+2*cos(\x)*sqrt(4-(sin(\x))^2))^.5*cos(\x)+(cos(\x)-sqrt(4-(sin(\x))^2))*sin(\x)))},
{(cos(\x)+sqrt(4-(sin(\x))^2))*sin(\x)
-((-2/3)*(cos(\x)-sqrt(4-(sin(\x))^2))*((55/16+cos(2*\x)+2*cos(\x)*sqrt(4-(sin(\x))^2))^.5*(sqrt(4-(sin(\x))^2))
+(3/4)*sin(\x))
*((1/4)*((12-cos(2*\x)+2*cos(\x)*sqrt(4-(sin(\x))^2))^.5*sin(\x)-(cos(\x)-sqrt(4-(sin(\x))^2))*cos(\x)))});
\fill (A) circle (1pt);
\fill (B) circle (1pt);
\fill (C) circle (1pt);
\draw (A)--(B)--(C)--(A);
}
\node at (1,-2.5) {\footnotesize A Poncelet Family};
\end{tikzpicture}\]
(As mentioned earlier, one of Schubert's goals was to enumerate
the families arising when the two circles are replaced by algebraic curves,
a generalization of this problem.)
It is clear that any triangle appears in some
Poncelet family, since we can always construct
an incircle and outcircle for a triangle. 
Since Chapple's formula is homogeneous,
these statements apply to similarity classes.

This paper started as an undergraduate research project
based on the idea of comparing the different Poncelet families 
by graphing them in a space of triangle classes we called the
{\it triangle of triangles}, which is the plane 
$\alpha+\beta+\gamma=\pi$ in the first octant of
$\alpha\beta\gamma$-space,
where $\alpha,\beta,\gamma$ are interior angles:
\[
\tdplotsetmaincoords{60}{120}
\begin{tikzpicture}[scale=1.75,tdplot_main_coords]
\coordinate (A) at (0,0,0); 
\coordinate (B) at (1,0,0); 
\coordinate (D) at (0,1,0); 
\coordinate (E) at (0,0,1); 

\draw[->,-latex] (A)--(1.7,0,0) node[left] {$\bs\alpha$};
\draw[->,-latex] (A)--(0,1.7,0) node[right] {$\bs\beta$};
\draw[->,-latex] (A)--(0,0,1.3) node[left] {$\bs\gamma$};
\fill (A) circle (.5pt);
\fill (B) circle (.5pt) node[left] {\footnotesize $(\pi,0,0)$};
\fill (D) circle (.5pt) node[right] {\footnotesize $(0,\pi,0)$};
\fill (E) circle (.5pt) node[left] {\footnotesize $(0,0,\pi)$};
\fill[yellow, opacity=0.5] (1,0,0) -- (0,1,0) -- (0,0,1) -- cycle;
\draw[thick] (B)--(D)--(E)-- cycle;
\node at (1,1,-.2) {\footnotesize The Triangle of Triangles};
\draw (.2,.9,.2) arc (270:180:.5 and .1) node[right] {\footnotesize $\bs\alpha+\bs\beta+\bs\gamma=\pi$};
\end{tikzpicture}
\]
By \cite[Theorem 293(b)]{J60} 
we find $r/R=4\sin(\alpha/2)\sin(\beta/2)\sin(\gamma/2)$. 
Since $r/R$ uniquely identifies a family, 
distinct Poncelet families
form distinct level curves for this function on the triangle of triangles,
drawn below. 
This image illustrates how distinct families
are related, and how the (disjoint) union of distinct
families partitions the triangle of triangles. 
The centroid corresponds to $r/R=1/2$,
where the incircle and outcircle are concentric, and the ``family'' consists of a
singleton, the equilateral class.
\begin{figure}[ht]
  \includegraphics[angle=-30,width=0.3\textwidth]{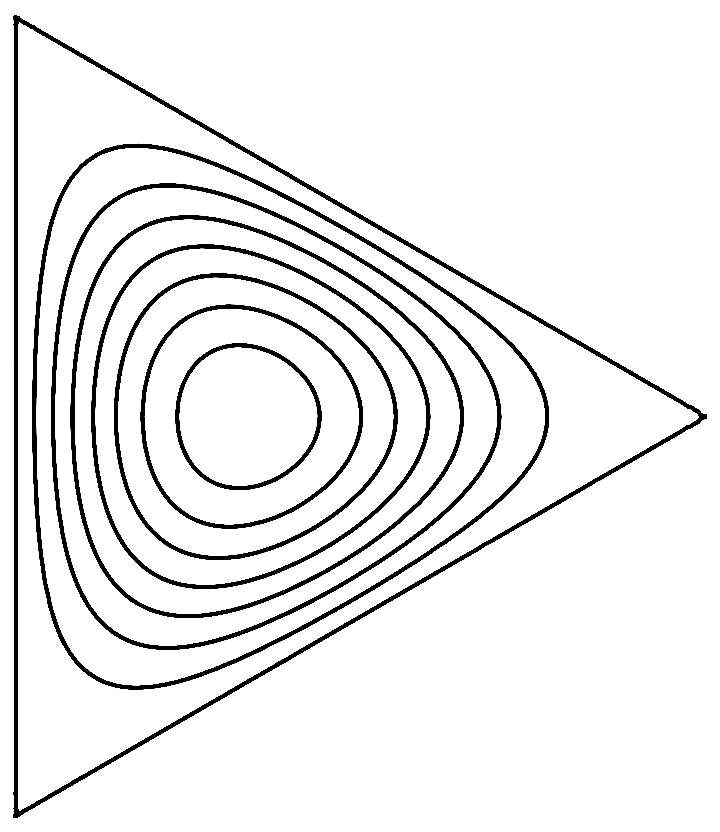}
\captionsetup{labelformat=empty, skip=-35pt}
\caption{\footnotesize  Family of Poncelet Families $\hphantom{qquad}$}
\end{figure}

The points on the border of the triangle of triangles are ``degenerate''
triangle classes, marked by exactly one zero-entry in 
the coordinates $(\alpha,\beta,\gamma)$, except at the corners, 
where there are two. We will call the degenerate points ``double points'',
because the corresponding triangle class has a vertex of multiplicity two.

\subsection{Collinear Triangles}
A (degenerate) triangle is {\it collinear} if its (nonzero) edges are collinear.
The triangle of triangles is good at describing Poncelet
families, since these ``inscribed'' classes are described completely by 
$(\alpha,\beta,\gamma)$.
But this variable consolidates all collinear classes into a single class,
since they are all $(0,0,0)\pmod\pi$, 
whereas some families 
degenerate to collinear classes that perhaps should be considered distinct:
Figure~\ref{constantsideratioprelim}
shows a family of isosceles classes on the left,
and a family with fixed side-length ratio of $(2:1)$ on the right.
\begin{equation}\label{constantsideratioprelim}
\begin{minipage}{.3\textwidth}
\centering
\begin{tikzpicture}[scale=.9]
\def\radius{1.3}
\coordinate (target) at ({\radius*cos(180)},{\radius*sin(180)}); 
\coordinate (target2) at ({\radius*cos(0)},{\radius*sin(0)}); 
\foreach \height in {-1.3,-.7,-.4,-.15,0,.15,.4,.7,1.3}{
\coordinate (vertex) at (0,{\height});
\fill (vertex) circle (1pt);
\fill (-\radius,0) circle (1pt);
\draw (vertex) -- (target);
\draw (vertex) -- (target2);
\draw[dashed] (0,-1.3)--(0,1.3);
\draw[thick] (target)--(target2);
}
\fill (target) circle (1.5pt);
\fill (target2) circle (1.5pt);
\fill (0,0) circle (1.5pt);
\node[below right] at (target2) {\footnotesize $C$};
\node[below left] at (target) {\footnotesize $B$};
\node[above] at ({\radius*cos(90)},{\radius*sin(90)}) {\footnotesize $A$};
\node at (0,-2) {\footnotesize Family of Isosceles Classes};
\end{tikzpicture}
\end{minipage}\qquad
\begin{minipage}{.5\textwidth}
\begin{tikzpicture}[scale=.9]
\def\radius{1.3}
\draw[line width=0.2mm, domain=-180:180, smooth, variable=\t, dashed] plot ({\radius*cos(\t)},{\radius*sin(\t)}); 
\coordinate (target) at ({-2*\radius},0); 
\coordinate (target2) at ({-\radius/2},0); 
\foreach \angle in {-170,-160,-140,-110,0,-45,45,110,140,160,170}{
\coordinate (vertex) at ({\radius*cos(\angle)},{\radius*sin(\angle)});
\fill (vertex) circle (1pt);
\fill ({\radius},0) circle (.5pt);
\draw (vertex) -- (target);
\draw (vertex) -- (target2);
\draw[thick] (target)--(target2);
}
\fill (target) circle (1.5pt);
\fill (target2) circle (1.5pt);
\fill ({-\radius},0) circle (1.5pt);
\node[below right] at (target2) {\footnotesize $C$};
\node[below left] at (target) {\footnotesize $B$};
\node[above right] at ({\radius*cos(45)},{\radius*sin(45)}) {\footnotesize $A$};
\node at (-.65,-1.9) {\footnotesize Family of $(2:1)$ Edge-Ratio Classes};
\node at (0,1.7) {${}$};
\end{tikzpicture}
\end{minipage}
\end{equation}
These families are distinguished by different-but-constant side-length ratios.
The triangle of triangles identifies
the two bold collinear limits in each family.
A more sensitive model that records side-length ratios of 
triangles would separate the two limits. 

\section{\Large Space of Triangles}\label{dn}

In this section we construct a space of triangles that distinguishes collinear triangles as well 
as double-vertex triangles, and coincides with the Schubert-Semple and Fulton-MacPherson spaces.

\subsection{Construction}

\begin{Definition}\label{triplets}\rm
Let $X=X_1+X_2\i,Y=Y_1+Y_2\i,Z=Z_1+Z_2\i$ be coordinate functions on $\A^3(\C)$.
Let $\X=\V(X+Y+Z)\subset\A^3(\C)$, and $\X_\R=\V((X_1,X_2)+(Y_1,Y_2)+(Z_1,Z_2))\subset\A^6(\R)$.
The set of {\it labeled ordered triplets} in the plane is the set
\[\{((B_{01},B_{02}),(x_1,x_2,y_1,y_2,z_1,z_2))\}\;\subset\;\A^2(\R)\times\X_\R\]
where $(B_{01},B_{02})$ is a {\it basepoint} and $(x_1,x_2),(y_1,y_2),(z_1,z_2)$ are {\it side-vectors}.
In complex notation,
\[\{(B_0,(\bs x,\bs y,\bs z))\}\;\subset\;\A^1(\C)\times\X\]
\[\begin{tikzpicture}
\node (A) at (1,0){};
\node (B) at (3,0) {};
\node (C) at (2.5,1.5) {};
\fill (A) circle (1pt);
\fill (B) circle (1pt);
\fill (C) circle (1pt);
\draw[thick, ->,-latex](B)--(C);
\draw[thick, ->,-latex] (C)--(A);
\draw[thick, ->,-latex] (A)--(B);
\node[below] at (2,0) {\footnotesize $\bs z$};
\node[left] at (1.75,.75) {\footnotesize $\bs y$};
\node[right] at (2.75,.75) {\footnotesize $\bs x$};
\node[below] at (B) {\footnotesize $B_0$};
\end{tikzpicture}\]
\end{Definition}
These are varieties in $\A^2(\R)\times\A^6(\R)$ and $\A^1(\C)\times\A^3(\C)$, respectively.
The sets of collinear, doubled, and tripled triplets form smooth real subvarieties of $\A^2(\R)\times\X_\R$
defined by 
\begin{align*}
T_\coll=\text{collinear triplets}&=\V(X_1 Y_2-X_2 Y_1,X_1 Z_2-X_2 Z_1,Y_1 Z_2-Y_2 Z_1)\\
T_\dbl=\text{doubled triplets}&=\V((X_1,X_2)(Y_1,Y_2)(Z_1,Z_2))\\
T_\tpl=\text{tripled triplets}&=\V(X_1,X_2,Y_1,Y_2,Z_1,Z_2)
\end{align*}
We say a triplet is {\it nondegenerate} if it is not collinear, doubled, or tripled,
and we identify these in the obvious way with oriented, labeled triangles of nonzero area in the plane.
The doubled triplets correspond naturally to the large diagonal on $\A^3(\C)$,
and the tripled triplets to the small diagonal.

The side-vectors already encode side-length ratios for the (degenerate) collinear triplets.
When a family of nondegenerate triplets approaches a doubled or tripled triplet,
the corresponding triangle degenerates in a ``collision''. 
To preserve the ``infinitesimal'' collision data we will blow up $\A^2(\R)\times\X_\R$ 
at $T_\dbl$ and $T_\tpl$.
We call the resulting space $\TT$, the space of 
{\it labeled, oriented, possibly-degenerate triangles},
and denote the strict transforms of the degenerate points by $\TT_\coll$, $\TT_\dbl$, and $\TT_\tpl$.

\begin{Theorem}\label{main1}
The iterated blow-up 
\[\TT:=\Bl_{\wt T_\dbl}\big(\Bl_{T_\tpl}(\A^2(\R)\times\X_\R)\big)\]
is a smooth $\R$-variety in $\A^2(\R)\times\X_\R\times\P(\X_\R)\times\P^1(\R)^3$,
where $\wt T_\dbl$ is the strict transform of $T_\dbl$,
and $\P(\X_\R)$ the real projective space. 
Its points 
\begin{equation}\label{thevar}
\big((B_{01},B_{02});(x_1,x_2,y_1,y_2,z_1,z_2);
[a_1,a_2,b_1,b_2,c_1,c_2];([u_1,u_2],[v_1,v_2],[w_1,w_2])\big)
\end{equation}
are characterized by the following conditions.
\begin{enumerate}[\rm i)]
\item\label{i}
If $(x_1,x_2,y_1,y_2,z_1,z_2)\neq 0$ then 
$[x_1,x_2,y_1,y_2,z_1,z_2]=[a_1,a_2,b_1,b_2,c_1,c_2]$.
\item\label{ii}
$a_1+b_1+c_1=a_2+b_2+c_2=0$. 
\item\label{iii}
If $(a_1,a_2)\neq 0$ then $[a_1,a_2]=[u_1,u_2]$,
and similarly for the others.
\item\label{iv}
If $(a_1,a_2)=0$ then $[u_1,u_2]$ is arbitrary, and
similarly for $[v_1,v_2],[w_1,w_2]$.
\end{enumerate}
\end{Theorem}

\begin{proof}
The first blow-up $\Bl_{T_\tpl}(\A^2(\R)\times \X_\R)$
is a subvariety of $\A^2(\R)\times\X_\R\times\P(\X_\R)$ by the blow-up construction.
It is smooth by \cite[Theorem 22.3.10]{Vakil} 
or \cite[Theorem 8.1.19]{Liu},
since $T_\tpl\subset\A^2(\R)\times \X_\R$ is a smooth embedding. 

Since $(x_1,x_2)+(y_1,y_2)+(z_1,z_2)=0$ on $\A^2(\R)\times\X_\R$, exactly one
alphabetical pair is zero on $T_\dbl$ away from $T_\tpl$.
Since the three smooth components $\V(X_1,X_2)$,
$\V(Y_1,Y_2)$, and $\V(Z_1,Z_2)$ have normal crossings at the origin,
the strict transform
$\wt T_\dbl$ is a disjoint union of their strict transforms.
Therefore $\wt T_\dbl$ is smooth, so
the second blow-up $\Bl_{\wt T_\dbl}(\Bl_{T_\tpl}(\A^2(\R)\times \X_\R))$
is smooth as well. 
We conclude $\TT$ is a smooth $\R$-variety.

To prove $\TT$ is characterized by \eqref{i}-\eqref{iv}, we'll show first that 
$\Bl_{\wt T_\dbl}(\Bl_{T_\tpl}(\A^2(\R)\times \X_\R))$
is cut out by the closed conditions
\begin{enumerate}[\rm a)]
\item\label{a}
$x_1+y_1+z_1=x_2+y_2+z_2=0$;
\item\label{b}
$(x_1,x_2,y_1,y_2,z_1,z_2)\left(\rk\sdbmx{x_1&x_2&y_1&y_2&z_1&z_2\\a_1&a_2&b_1&b_2&c_1&c_2}=1\right)=0$.
\item\label{c}
$\sdvmx{a_1&a_2\\u_1&u_2}=\sdvmx{b_1&b_2\\v_1&v_2}=\sdvmx{c_1&c_2\\w_1&w_2}=0$.
\end{enumerate}
The relation \eqref{a} defines $\X_\R$ inside $\A^6(\R)$.
The coordinates satisfy \eqref{b} and \eqref{c} by
the description of the coordinate rings, for example in \cite[Chapter 8]{Liu}, as follows:
The relation \eqref{b} expresses the fact that the first blow-up is the graph
of the rational function 
\[[x_1,x_2,y_1,y_2,z_1,z_2]:\A^2(\R)\times\X_\R\,\dashlr\,\P(\X_\R)\,\subset\P^5(\R)\]
which is regular away from $T_\tpl$. 
For \eqref{c}, the strict transform $\wt T_\dbl$ is defined by
$(a_1,a_2)(b_1,b_2)(c_1,c_2)=0$, or the vanishing of 
the eight monomials $a_ib_jc_k$ on $[a_1,a_2,b_1,b_2,c_1,c_2]$.
Therefore the final blow-up is the graph of the rational function 
\[
\{(a_ib_jc_k)\}:\Bl_{T_\dbl}(\A^2(\R)\times\X_\R\times\P(\X_\R))\,\dashlr\,\P^7(\R)
\] 
whose restriction to $\P^5(\R)$ is the Segre embedding, as in \cite[Example 2.11]{Ha92},
given away from $\wt T_\dbl$ by $([a_1,a_2],[b_1,b_2],[c_1,c_2])\in\P^1(\R)^3$.
Thus $\TT\,\subset \A^2(\R)\times\X_\R\times\P(\X_\R)\times\P^1(\R)^3$,
and the last component consists of triples $([u_1,u_2],[v_1,v_2],[w_1,w_2])$ satisfying
the closed condition \eqref{c}.

It remains to prove \eqref{a},\eqref{b}, and \eqref{c} are equivalent to 
\eqref{i}-\eqref{iv}.
\eqref{i} follows from \eqref{b},
\eqref{ii} follows from \eqref{a} and \eqref{b},
\eqref{iii} from \eqref{c}, and \eqref{iv} from \eqref{c}.
Conversely, \eqref{ii} implies \eqref{a} if the $\A^6(\R)$-component is nonzero,
so \eqref{ii} implies \eqref{a} regardless.
Similarly, if the $\A^6(\R)$-component is nonzero then \eqref{i} implies the rank condition in \eqref{b},
so \eqref{i} implies \eqref{b}.
Finally, \eqref{iii} and \eqref{iv} together evidently imply \eqref{c}.
\end{proof}

\subsubsection*{{\rm 3.1.1.} Isomorphism with Fulton-MacPherson Space}
In \cite[Section 1]{FM94} the $n$-point configuration space $X[n]$ is described in different ways:
In ``inductive construction'' as the result of a sequence of blow-ups,
and in ``first geometric description'' as a point-set, using screens.
We use the first to give a natural isomorphism $\TT\isom\FM{3}$. 
Briefly,  by \cite[p.196]{FM94},
$\FM{3}$ is the iterated blowup of $\A^3(\C)$ first at the small diagonal,
which corresponds to $T_\tpl$,
and then at the strict transform of the large diagonal, which is the strict transform of $T_\dbl$.
In Subsection~\ref{corresp} we will
give the point-theoretic correspondence between $\TT$ and $\FM{3}$ using the first geometric description.

\begin{Proposition}
The space $\TT$ of Theorem~\ref{main1} is isomorphic to the Fulton-MacPherson space
$\FM{3}$ and to the Schubert-Semple space in \cite{Semple}.
\end{Proposition}

\begin{proof}
We use complex notation.
The map $(B_0;(\bs x,\bs y,\bs z))\lm (B_0-\bs z,B_0,B_0+\bs x)$ assigns vertices to the triplets in Definition~\ref{triplets}, 
and defines a natural isomorphism $\varrho:\A^2(\R)\times\X_\R\lr\A^2(\R)^3$.
Set $\Delta_\tpl=\varrho(T_\tpl)$ and $\Delta_\dbl=\varrho(T_\dbl)$, then
$\Delta_\tpl=\{(B_0,B_0,B_0)\}$ is the small diagonal, and $\Delta_\dbl=\{(A,B,C):A=B, B=C,\text{ or }A=C\}$
is the large diagonal.
By \cite[p.196]{FM94}, $\FM{3}=\Bl_{\wt\Delta_\dbl}(\Bl_{\Delta_\tpl}(\A^2(\R)^3)$,
where $\wt\Delta_\dbl$ is the strict transform, and this is evidently isomorphic 
to our space $\TT=\Bl_{\wt T_\dbl}(\Bl_{T_\tpl}(\A^2(\R)\times\X_\R)$, by e.g.
\cite[Proposition 1.15, Section 8.1.2]{Liu}.
The last statement follows since
$\FM{3}$ is isomorphic to the Schubert-Semple space $W_6^*$ in \cite[Theorem II]{Semple} 
by \cite[p.189]{FM94}. 
\end{proof}

\begin{Remark}
Though the spaces are isomorphic, we do not prove that the families implied by Definition~\ref{triangles} coincide
with those defined using screens in \cite{FM94}.
\end{Remark}

\subsection{Possibly-Degenerate Triangles}
\begin{Definition}\label{triangles}\rm
\noindent a)
We call $\TT$ the space of {\it labeled, oriented, possibly-degenerate triangles}.

\noindent b)
The {\it nondegenerate locus} is
\[\TT_\ndgen:=\left\{p\in\TT: \rk\sdbmx{x_1&y_1&z_1\\x_2&y_2&z_2}=2\right\}\]
A nondegenerate $p\in\TT$ is determined completely by the data $((B_{01},B_{02});(x_1,x_2,y_1,y_2,z_1,z_2))$, which is equivalent to the data of three 
distinct non-collinear points in the plane.
The
{\it degenerate boundary} is the subvariety \[\TT_\dgen:=\TT-\TT_\ndgen
=\left\{p\in\TT:\rk\sdbmx{x_1&y_1&z_1\\x_2&y_2&z_2}<2\right\}\] 
which stratifies into the subvarieties
\begin{enumerate}[\rm i)]
\item
$\TT_\coll$, the {\it collinear points}, given by $\rk\,\sdbmx{u_1&v_1&w_1\\ u_2&v_2&w_2}=1$.
\item
$\TT_\dbl$, the {\it double points},  given by $(a_1,a_2)(b_1,b_2)(c_1,c_2)=0$.
\item
$\TT_\tpl$, the {\it triple points},  given by $(x_1,x_2,y_1,y_2,z_1,z_2)=0$.
\end{enumerate}
\end{Definition}

\begin{equation}\label{types}
\begin{minipage}{0.3\textwidth}
\begin{tikzpicture} 
\node (A) at (1,0){};
\node (B) at (3,0) {};
\node (C) at (2.5,1.5) {};
\fill (A) circle (1pt);
\fill (B) circle (1pt);
\fill (C) circle (1pt);
\draw[->,-latex](B)--(C);
\draw[->,-latex] (C)--(A);
\draw[->,-latex] (A)--(B);
\node[below] at (A) {\footnotesize $B$};
\node[below] at (B) {\footnotesize $C$};
\node[right] at (C) {\footnotesize $A$};
\node at (2.2,-.8) {\footnotesize Nondegenerate Point};
\end{tikzpicture}
\end{minipage}
\begin{minipage}{0.3\textwidth}
\begin{tikzpicture} 
\node (A) at (1,0){};
\node (B) at (3,0) {};
\node (C) at (4,0) {};
\fill (A) circle (1pt);
\fill (B) circle (1pt);
\fill (C) circle (1pt);
\draw[->,-latex]([yshift=1]2.9,0)--([yshift=1]1.1,0);
\draw[->,-latex]([yshift=1]3.9,0)--([yshift=1]3.1,0);
\draw[->,-latex]([yshift=-2]1.1,0)--([yshift=-1]3.9,0);
\node[below] at (A) {\footnotesize $B$};
\node[below] at (B) {\footnotesize $A$};
\node[below] at (C) {\footnotesize $C$};
\node at (2.5,-.8) {\footnotesize Collinear Point};
\fill (2.5,-1.8) circle (1pt);
\draw (2.5,-1.8) circle (2pt);
\draw (2.5,-1.8) circle (3pt);
\node at (2.5,-2.3) {\footnotesize Triple Point};
\end{tikzpicture}\end{minipage}
\begin{minipage}{0.3\textwidth}
\begin{tikzpicture} 
\node (A) at (1,0){};
\node (B) at (3,0) {};
\fill (A) circle (1pt);
\fill (B) circle (1pt);
\draw (B) circle (2pt);
\draw[->,-latex]([yshift=1.2]2.9,0)--([yshift=1.2]1.1,0);
\draw[->,-latex]([yshift=-1.2]1.1,0)--([yshift=-1.2]2.9,0);
\draw[->,-latex] (B)--({3+.5*cos(45)},{.5*sin(45)});
\draw[->,-latex] (B)--({3+.5*cos(135)},{.5*sin(135)});
\draw[->,-latex] (B)--({3+.5*cos(0)},{.5*sin(0)});
\draw[->,-latex] (B)--({3+.5*cos(90)},{.5*sin(90)});
\node[below] at (A) {\footnotesize $B$};
\node[below] at (B) {\footnotesize $A=C$};
\node at (2,-.8) {\footnotesize Double Point};
\end{tikzpicture}
\end{minipage}
\end{equation}

\begin{Remark}\label{retract}
The triple points $\TT_\tpl$ are a deformation
retract of $\TT$ under the homotopy
\begin{align*}
H:\TT\times[0,1]&\lr \TT\\
(p,t)&\lm p_t
\end{align*}
relative to $\TT_\tpl$, where the components of $p$ and $p_t$ are
identical except for the $\X_\R$-component,
which for $p_t$ is scaled by $1-t$. 
Thus $H|_{\TT_\tpl}=\id$, 
$H(p,0)=p$, and $H(p,1)=p_1$ is in $\TT_\tpl$.
\end{Remark}

The degenerate triangle types specialize as follows:
\[\begin{tikzcd}& &\TT\arrow[dl,-latex]\arrow[d,-latex]\arrow[dr,-latex] &\\
&\TT_\coll\arrow[d,-latex]\arrow[dr,-latex]&\TT_\tpl\arrow[dl,-latex]\arrow[dr,-latex]&\TT_\dbl\arrow[dl,-latex]\arrow[d,-latex]\\
&\TT_\coll\cap\TT_\tpl\arrow[dr,-latex]&\TT_\coll\cap\TT_\dbl\arrow[d,-latex]&\TT_\tpl\cap\TT_\dbl\arrow[dl,-latex]\\
& &\TT_\coll\cap\TT_\tpl\cap\TT_\dbl &
\end{tikzcd}\]
This list resembles the one in \cite[Section 3]{Semple} and \cite[Section 1]{ColFul89}.
The variable of Schubert, by way of Semple, starts as the $6$-tuple of projective vertices and lines,
$((A,B,C),(a,b,c))$, and adds a ``triangle variable'' in Section 4, which amounts to
a (possibly degenerate) circumscribing conic that
models different approaches to the (singular) locus of $\coll$-$\dbl$-$\tpl$-points.

\subsection{Orientation}
\begin{Definition}\rm
A nondegenerate triangle has {\it positive orientation} if the directed graph
defined by its side-vectors $(\a,\b,\c)=(x_1+x_2\i,\,y_1+y_2\i,\,z_1+z_2\i)$
goes counterclockwise in the plane, and {\it negative orientation} if it goes clockwise. 
A degenerate triangle has {\it zero orientation}.
\end{Definition}

\begin{Definition}\label{vertangles}\rm 
Let $p\in\TT$ be  a labeled, oriented, possibly-degenerate triangle.
In complex notation we write $[\bs t]=[t_1+t_2\i]$ for $[t_1,t_2]\in\P^1(\R)$, and
\[p=\left(B_0;(\bs x,\bs y,\bs z);[a_1,a_2,b_1,b_2,c_1,c_2];([\u],[\v],[\w])\right)\]

\noindent a)
The {\it vertices} of $p$ are $(A,B,C)=(B_0-\br{z_1,z_2},B_0,B_0+\br{x_1,x_2})$.

\noindent b)
Define $\rho:\P^1(\R)\isim\R/\pi$ by $\rho[\bs z]=\Arg(\bs z)\pmod\pi$, where $\Arg$ is the 
principal argument; the unique real number in the interval $(-\pi,\pi]$ such that $\bs z=|\bs z|e^{\Arg(\bs z)\i}$.
Then 
\begin{enumerate}[\rm i)]
\item
The {\it arguments} of $p$
are $\xi_\u=\rho([\u])$,
$\xi_\v=\rho([\v])$, and $\xi_\w=\rho([\w])$.
\item
The {\it interior angles} $(\alpha,\beta,\gamma)\in\A^3(\R/\pi)$ of $p$ are given by the ``vector product'' relation
\begin{equation}\label{intangles}
\br{\alpha,\beta,\gamma}=\br{\xi_\u,\xi_\v,\xi_\w}\times \br{1,1,1}=\br{\xi_\v-\xi_\w,\,\xi_\w-\xi_\u,\,\xi_\u-\xi_\v}\;\pmod\pi
\end{equation}
\item
The interior angles of $p$ satisfy $\alpha+\beta+\gamma=0\pmod\pi$, hence form a subvariety
\[\T:=\V(\alpha+\beta+\gamma)\subset\A^3(\R/\pi)\]
\end{enumerate}
\end{Definition}

\begin{Remark}
The set $\TT_\coll$ consists of those $p\in\TT$ with $\xi_\u=\xi_\v=\xi_\w$,
or equivalently $\alpha=\beta=\gamma=0$.
By
Lemma~\ref{rep} below, $p\in\TT_\ndgen$ has positive orientation if its interior angles have representatives in $[0,\pi)$
adding to $\pi$, and
negative orientation if they have representatives in $(-\pi,0]$ adding to $-\pi$.
\[
\qquad\qquad\begin{minipage}{.4\textwidth}
\centering
\begin{tikzpicture}[scale=.65]
\draw[->,-latex] (0,0)--(3.5,1); 
\draw[->,-latex] (3.5,1)-- (2,3); 
\draw[->,-latex] (2,3)--(0,0); 
\node[right] at (3.5,1) {$C$};
\node[above] at (2,3) {$A$};
\node[left] at (0,0) {$B$};
\node[right] at (2.8,2) {$\b$};
\node[left] at (1,1.5) {$\c$};
\node[below] at (1.75,.5) {$\a$};
\node at (.7,.5) {$\beta$};
\node at (2,2.3) {$\alpha$};
\node at (2.8,1.2) {$\gamma$};
\draw[domain=atan(1/3.5):atan(3/2),->,-latex] plot ({.6*cos((\x))},{.6*sin((\x))}); 
\draw[domain=-123.690068:atan(-4/3),->,-latex] plot ({2+.5*cos((\x))},{3+.5*sin((\x))}); 
\draw[domain=126.869898:195.945396,->,-latex] plot ({3.5+.5*cos((\x))},{1+.5*sin((\x))}); 
\node at (1.75,-.5) {\footnotesize Positive Orientation};
\end{tikzpicture}
\end{minipage}
\begin{minipage}{.4\textwidth}
\centering
\begin{tikzpicture}[scale=.65]
\draw[->,-latex] (3.5,1)--(0,0);
\draw[->,-latex] (2,3)--(3.5,1);
\draw[->,-latex] (0,0)--(2,3);
\node[right] at (3.5,1) {$B$};
\node[above] at (2,3) {$A$};
\node[left] at (0,0) {$C$};
\node[right] at (2.8,2) {$\c$};
\node[left] at (1,1.5) {$\b$};
\node[below] at (1.75,.5) {$\a$};
\node at (.7,.5) {$\gamma$};
\node at (2,2.3) {$\alpha$};
\node at (2.8,1.2) {$\beta$};
\node at (1.75,-.5) {\footnotesize Negative Orientation};
\draw[domain=atan(3/2):atan(1/3.5),->,-latex] plot ({.6*cos((\x))},{.6*sin((\x))});
\draw[domain=atan(-4/3):-123.690068,->,-latex] plot ({2+.5*cos((\x))},{3+.5*sin((\x))}); 
\draw[domain=195.945396:126.869898,->,-latex] plot ({3.5+.5*cos((\x))},{1+.5*sin((\x))}); 
\end{tikzpicture}
\end{minipage}
\]
\end{Remark}

\section{\Large Space of Similarity Classes}\label{triangleclasses}

\subsection{Action of Similarity Group on $\TT$}\label{action}
There is a natural group action 
\[\Sim^+(2)\times\TT\lr\TT\]
where $\Sim^+(2)=\C\rtimes\C^\times$ is the group of
orientation-preserving similarities.
Explicitly, let $p\in\TT$ be given in complex notation by $(B_0;(\bs x,\bs y,\bs z);[\bs a,\bs b,\bs c];([\bs u],[\bs v],[\bs w]))$,
and fix $(\bs\lambda,\bs\mu)\in\C\rtimes\C^\times$. Then
\begin{enumerate}[$\circ$]
\item
$(\bs\lambda,\bs\mu)\cdot B_0=\bs\mu B_0+\bs\lambda$ on the $\A^2(\R)$-component.
\item
$(\bs\lambda,\bs\mu)\cdot (\bs x,\bs y,\bs z)=(\bs\mu\bs x,\bs\mu\bs y,\bs\mu\bs z)$
on the $\X_\R$-component.
\item
$(\bs\lambda,\bs\mu)\cdot [\bs a,\bs b,\bs c]=[\bs\mu\bs a,\bs\mu\bs b,\bs\mu\bs c]$
on the $\P(\X_\R)$-component.
\item
$(\bs\lambda,\bs\mu)\cdot[\bs u]=[\bs\mu\bs u]$, 
$(\bs\lambda,\bs\mu)\cdot[\bs v]=[\bs\mu\bs v]$, 
and $(\bs\lambda,\bs\mu)\cdot[\bs w]=[\bs\mu\bs w]$ on the $\P^1(\R)^3$-component.
\end{enumerate}

\begin{Definition}\label{similarity}\rm
We say labeled, oriented, possibly-degenerate triangles $p,p'\in\TT$ are {\it similar},
and write $[p]=[p']$, if they are equivalent modulo $\Sim^+(2)$.
\end{Definition}

\begin{Remark}
Let $H:\TT\times[0,1]\lr\TT$ be the homotopy of Remark~\ref{retract},
with $p_t=H(p,t)$. If $p\not\in\TT_\tpl$, the set $\{p_t\}$ is a continuous 
(infinite) family,
and though $[p]=[p_t]$ for all $t\neq 1$, we have $[p]\neq[p_1]$.
Since the preimage of $[p]$ in $\TT$ contains $\{p_t:t\neq 0\}$
but does not contain the limit point $p_1$, it is not closed, 
hence the singleton $\{[p]\}$ is not closed in the quotient topology.
On the other hand, any open set of $\TT$ containing $p_1$
contains some $p_t$, hence any open set containing $[p_1]$ contains $[p]$. 
Thus $[\TT]=\TT/\Sim^+(2)$ is topologically not even $T_1$.
We remedy this situation as follows.
\end{Remark}

\subsection{The Quotient $\DD$ of $\TT-\TT_\tpl$ by Similarity}

\begin{Definition}\label{DD}\rm
Denote $\TT^\circ:=\TT-\TT_\tpl$ and $\DD:=\TT^\circ/\Sim^+(2)$.
\end{Definition}

\begin{Lemma}\label{orbits}
The orbit space $\DD$ is naturally parameterized as
\[\DD=\left\{([\a,\b,\c]; (\alpha,\beta,\gamma))\right\}
\subset\P(\X)\times\T\]
where $\T$ is the space of interior angles of $p\in\TT$
in Definition~\ref{vertangles}(b)(iii),
$\a=a_1+a_2\i,\b=b_1+b_2\i,\c=c_1+c_2\i$ in the notation of
Theorem~\ref{main1}, and each pair $([\a,\b,\c];(\alpha,\beta,\gamma))$ 
satisfy the conditions imposed by points of $\TT^\circ$ in Theorem~\ref{main1}.
\end{Lemma}

\begin{proof}
The action of $\Sim^+(2)$ on $\A^2(\R)$ is transitive, so all basepoints
are identified in the quotient.
The orbits on $\X_\R^\circ:=\X_\R-\{0\}$, given by nonzero
complex scaling, form the complex projective space 
$\P(\X)=\{[\a,\b,\c]:\a+\b+\c=\0\}$, which is well-defined since
the origin is excluded.
The $\P(\X_\R)$-component of $\TT^\circ$ is determined by the $\X_\R$-component
by Theorem~\ref{main1}(i), so its orbit space is the same, 
and we can write $[\a,\b,\c]$ for a point in either.
We use the letters $a,b,c$ 
over $x,y,z$ because the $\P(\X_\R)$-component gives
the correct data on triple points.

To compute $\P^1(\R)^3/\C^\times$,
let $\rho:\P^1(\R)\lr\R/\pi$ be the isomorphism of Definition~\ref{vertangles}(b),
and let $(\xi_\u,\xi_\v,\xi_\w)\in\A^3(\R/\pi)$ be the triple of arguments of
$p\in\TT^\circ$.
Multiplication by $\bs\mu$ on the arguments
additively translates each by $\Arg(\bs\mu)$.
Let $\Delta=\{(\xi,\xi,\xi)\}$ denote
the diagonal of $\A^3(\R/\pi)$, and put $\P(\A^3(\R/\pi))=\A^3(\R/\pi)/\Delta$. 
Then $\rho$ followed by the interior-angles map \eqref{intangles}
defines an isomorphism
\begin{equation}\label{isom}
\begin{aligned}
\P^1(\R)^3/\C^\times\isim &\P(\A^3(\R/\pi))\isim\T \\
[[\bs u],[\bs v],[\bs w]]\lm &{[\xi_\u,\xi_\v,\xi_\w]}\lm(\alpha,\beta,\gamma)
\end{aligned}
\end{equation}
Thus the orbit space of $\P^1(\R)^3$ may be identified with $\T$.
Putting
these together yields the result.
\end{proof}

\begin{Remark}
On $\TT^\circ$ the action on the $\X_\R$-component factors through the $\P(\X_\R)$-component,
setting up an isomorphism between the non-triple point and 
the (deleted) triple point orbit spaces:
\[\DD=\TT^\circ/\Sim^+(2)\isom\TT_\tpl/\Sim^+(2)\]
\end{Remark}

\begin{Definition}\label{degclasses}\rm
Denote the degenerate classes on $\DD$ by 
\[
\DD_\coll=[\TT_\coll\cap\TT^\circ],\quad 
\DD_\dbl=[\TT_\dbl\cap\TT^\circ],\quad
\DD_\dgen=\DD_\coll\cup\DD_\dbl
\]
\end{Definition}

\subsection{$\DD$ as a Blow-Up}
Let $X,Y,Z$ be the coordinate functions on $\A^3(\C)$, so
$\X=\V(X+Y+Z)\subset\A^3(\C)$.
The variety $\V(X)\subset\X$ is a one-dimensional linear subspace, and its orbit 
is the point $p_X=[\0,\1,-\1]$. Similarly $p_Y=[\1,\0,-\1]$ and $p_Z=[\1,-\1,\0]$.
Denote the corresponding divisor of double points by \[D_\dbl=\{p_X,p_Y,p_Z\}\subset\P(\X)\] 
We emphasize that $\P(\X)_\R$ is the underlying $\R$-variety of $\P(\X)$, 
different from the real projective space $\P(\X_\R)$.

\begin{Theorem}\label{main2}
The orbit space $\DD$ is a smooth $\R$-variety isomorphic to
the real oriented blow-up $\Bl_{D_\dbl}(\P(\X)_\R)\isom \#_3\P^2(\R)$,
which is Dyck's surface,
and there is a commutative diagram
\[\begin{tikzcd}
&\TT^\circ\ar[d,"{[-]}"',-latex]\ar[r,-latex]
&\A^2(\R)\times\X_\R^\circ\ar[d,"{[-]}",-latex]
\\
&\DD\ar[r,-latex]&\P(\X)_\R
\end{tikzcd}\]
where $\X_\R^\circ=\X_\R-\{0\}$, and the vertical arrows are quotient-by-similarity.
\end{Theorem}

\begin{proof}
Since $\P(\X)_\R$ is isomorphic to $\P^1(\C)_\R$, which is simply connected, and $D_\dbl$ consists of
three distinct points on $\P(\X)_\R$, 
the real oriented blow-up is a smooth $\R$-variety isomorphic to $\#_3\P^2(\R)$, 
which is Dyck's surface.

Let $L\in\C[X,Y,Z]$ be a linear form that is not a (scalar) multiple of $X,Y,Z$, or $X+Y+Z$
so that $p_L,p_X,p_Y,p_Z$ are four distinct points on the (complex) line $\P(\X)$.
Let $X_L=\tfrac XL$, $Y_L=\tfrac YL$, and $Z_L=\tfrac ZL$ be the coordinate functions
on the distinguished affine open subset $\P(\X)_L$, and let 
$(X_{L,1},X_{L,2}),(Y_{L,1},Y_{L,2}),(Z_{L,1},Z_{L_2})$ be their real counterparts
on $\P(\X)_{L,\R}$.
Then $D_\dbl$ 
is the zero-set of the eight monomials $\{X_{L,i}Y_{L,j}Z_{L_k}\}$ on $\P(\X)_{L,\R}$,
and $\Bl_{D_\dbl}(\P(X)_{L,\R})$ is the graph of the rational function
\begin{align*}
\P(\X)_{L,\R}&\dashlr\P^1(\R)^3\;\raisebox{1pt}{\hlr}\P^7(\R)\\
p&\lm\left(\left[X_{L,1},X_{L,2}\right](p),\left[Y_{L,1},Y_{L,2}\right](p),\left[Z_{L,1},Z_{L,2}\right](p)\right)
\end{align*}
where the inclusion $\Sigma:\P^1(\R)^3\lr\P^7(\R)$ is the Segre map, as in the proof of Theorem~\ref{main1}.

A different choice of linear form $L'$ multiplies the components of $\P^1(\R)^3$
by the complex number $\bs\mu=L'(p)/L(p)$ as described in Subsection~\ref{action},
so that the atlas of affine charts induces the equivalence $\P^1(\R)^3/\C^\times$ in the proof of Lemma~\ref{orbits}.
Thus we obtain a well-defined rational morphism, 
regular on $\P(\X)_{L,\R}-D_\dbl$, written in complex notation as
\begin{align*}
\P(\X)_\R&\dashlr\P^1(\R)^3/\C^\times\\
[\a,\b,\c]&\lm[[\a],[\b],[\c]]
\end{align*}
and $\Bl_{D_\dbl}(\P(\X)_\R)$ is its graph.
On $\P(\X)_\R-D_\dbl$ the triple $[[\a],[\b],[\c]]\in\P^(\R)^3/\C^\times$ conforms to 
the closed conditions of
Theorem~\ref{main1}(iii,iv) on the $\P(\X_\R)\times\P^1(\R)^3$-component of $\TT$, 
hence these conditions extend to the fiber over $D_\dbl$ as well.
Composition with the map \eqref{isom} shows $\Bl_{D_\dbl}(\P(\X)_\R)$  
may be identified with $\DD$ as in Lemma~\ref{orbits}, as desired.
The commutativity of the diagram is immediate:
Let \[p=((B_{01},B_{02});(x_1,x_2,y_1,y_2,z_1,z_2);[a_1,a_2,b_1,b_2,c_1,c_2];([\u],[\v],[\w])\]
be a general element of $\TT^\circ$,
and let $(\alpha,\beta,\gamma)$ be the interior angles determined 
by $([\u],[\v],[\w])$ as in Definition~\ref{vertangles}.
Then since $p$ is not in $\TT_\tpl$, 
$(x_1,x_2,y_1,y_2,z_1,z_2)=(a_1,a_2,b_1,b_2,c_1,c_2)$, and following the maps through 
gives commutativity:
\[\begin{tikzcd}
&p\ar[d,"{[-]}"',-latex]\ar[r,-latex]
&((B_{01},B_{02});(a_1,a_2,b_1,b_2,c_1,c_2))\ar[d,"{[-]}",-latex]
\\
&([a_1+a_2\i,b_1+b_2\i,c_1+c_2\i]; (\alpha,\beta,\gamma))\ar[r,-latex] 
& {[a_1+a_2\i,b_1+b_2\i,c_1+c_2\i]}
\end{tikzcd}\]
\end{proof}

\begin{Remark}
The degenerates of Definition~\ref{degclasses} are the real subvarieties
$\DD_\coll=\V(\alpha,\beta,\gamma)$, $\DD_\dbl=\V(\a\b\c)$, and $\DD_\dgen=\V(\alpha\beta\gamma)$,
where we abuse notation in the obvious way.
The collinear classes constitute the strict transform of the (real) subvariety
$\rk_\R[\a,\b,\c]=1$ on $\P(\X)_\R$.
The double point classes form the three closed fibers of the blow-up at 
$D_\dbl=\{p_X,p_Y,p_Z\}$ on $\P(\X)_\R$.
\end{Remark}

\section{\Large Geometric Explanations}\label{discussion}
In this section we 
describe the exact correspondence between $\TT$ and the Fulton-MacPherson space $\FM{3}$ of \cite{FM94} on the level of points. 
Then we
explain the geometry of the free $\P^1(\R)$-entry at $\TT$'s double points, 
and justify our inclusion of orientation.

\subsection{Point Correspondence with Fulton-MacPherson Space}\label{corresp}

Define ``pure'' versions
\begin{align*}
\TT_\coll^\circ&=\TT_\coll-(\TT_\coll\cap(\TT_\dbl\cup\TT_\tpl))\\
\TT_\dbl^\circ&=\TT_\dbl-(\TT_\dbl\cap\TT_\tpl)\\
\TT_\tpl^\circ&=\TT_\tpl-(\TT_\tpl\cap\TT_\dbl)
\end{align*}
Then 
$\TT=\TT_\ndgen\sqcup \TT_\coll^\circ\sqcup\TT_\dbl^\circ\sqcup\TT_\tpl^\circ
\sqcup(\TT_\tpl\cap\TT_\dbl)$.
We use complex notation and write the variable \eqref{thevar}'s four components as
\[p=(B_0;(\bs x,\bs y,\bs z);(a_1,a_2,b_1,b_2,c_1,c_2);([\u],[\v],[\w]))
\in\A^2(\R)\times\X_\R\times\P(\X_\R)\times\P^1(\R)^3\]
The first two components determine three vertices in the plane
by Definition~\ref{vertangles}(a), hence a (possibly degenerate) configuration
of three points in the sense of \cite{FM94}. 

The possibly-degenerate configurations of $n$-point configurations are enumerated and sorted in 
\cite[Section 1]{FM94} by {\it screens.}
Screens are designed to expand out the infinitesimal behavior of degenerate
configurations. Briefly, 
every point $p$ of multiplicity $m>1$ in a configuration is a basepoint for a tree of screens.
The first screen at $p$ 
is an equivalence class of configurations of $m$ points in the tangent space at $p$,
where equivalence is defined by translation and (nonzero) real scaling.
The first screen must have at least two distinct points,
and can have points with nontrivial multiplicities.
Any point of nontrivial multiplicity becomes the basepoint for a subsequent screen, 
and the process continues until reaching a screen on which all points have multiplicity one.

Thus every point of multiplicity $m$ in the original configuration has a finite tree
of at most $m-1$ screens.
The totality of this data, an original configuration and its trees of screens
at points of nontrivial multiplicity,
separates and classifies possibly-degenerate configurations by \cite{FM94},
producing a moduli space of screened point-configurations \cite[Theorem 4]{FM94}.

The configuration of points (of possibly nontrivial multiplicity) on a screen 
approach the centroid as they are scaled down, at relative rates measured by their relative distances to
the centroid. 
If there are only two points, the approach rates are the same, and the screen
then captures multiplicity and approach direction.
Thus the screen
encodes multiplicity, ``approach directions'', and if there are more than two points, ``approach rates''.

Three-point configurations have either no screens,
a single screen at a double or triple
point, or two screens at a triple point.
We list the screen types and their correspondence with our variable
case-by-case:
\begin{enumerate}[\rm (1)]
\item
$\TT_\ndgen\sqcup\TT_\coll^\circ\longleftrightarrow\{\text{configurations with no screens}\}$: 
In $\TT$ this is the case $\bs x\bs y\bs z\neq\0$, and the configurations are in one-to-one correspondence
with the values of $(B_0;(\bs x,\bs y,\bs z))$, which then determine the other components
by Theorem~\ref{main1}.
\item
$\TT_\dbl^\circ\longleftrightarrow
\{\text{configurations with one screen with two distinct points}\}$:
A point of $\TT_\dbl^\circ$ in $\TT$ is the case $(\bs x,\bs y,\bs z)\neq(\0,\0,\0)$ and $\bs x\bs y\bs z=\0$,
so that $(B_0;(\bs x,\bs y,\bs z))$ determines the other components except for a free
$\P^1(\R)$-variable in the same position in the $\P^1(\R)^3$-component as the zero 
alphabetical side-vector in the $\X_\R$-component. 
The first two components uniquely determine the location of the
single point and the double point, and the free variable parameterizes the approach direction
at the double point encoded by the single screen.
\item
$\TT_\tpl^\circ\longleftrightarrow\{\text{configurations with one screen with three distinct points}\}$:
In $\TT$ this is $(\bs x,\bs y,\bs z)=(\0,\0,\0)$ and $(a_1,a_2)(b_1,b_2)(c_1,c_2)\neq 0$,
so the $\P^5(\R)$-component determines the $\P^1(\R)^3$-component.
The first component gives the location of the triple point, and the 
$\P^5(\R)$-component, which consists of three distinct nonzero side-vectors $(a_1,a_2),(b_1,b_2),(c_1,c_2)$ up to (simultaneous) real scaling,
parameterize the approach directions and relative approach rates of the three points on the screen, which converge
(at the centroid) to the triple point.
\item
$\TT_\tpl\cap\TT_\dbl\longleftrightarrow\{\text{configurations with two screens of one double point and two (distinct) points}\}$: 
In $\TT$ this is the case $(\bs x,\bs y,\bs z)=(\0,\0,\0)$ and $(a_1,a_2)(b_1,b_2)(c_1,c_2)=0$,
and there is a free argument in the $\P^1(\R)^3$-component corresponding to the double point,
which is given by the zero alphabetical pair.
The first component gives the location of the triple point,
the $\P^5(\R)$-component, whose alphabetical pairs determine a 
nonzero side-vector and its negative, up to real scaling,
give the approach direction of the single and double point to the triple point,
on the first screen. The free variable in the $\P^1(\R)^3$-component gives the approach
direction to the double point on the second screen.
\end{enumerate}


\subsection{Degenerations and Double Point Arguments}\label{doublepoints}
In Theorem~\ref{main1}, $[\v]=[v_1,v_2]$
can take any value at a double point $(b_1,b_2)=0$.
Figure \eqref{family2} depicts a family of triangles
\[\{(B_0;(\a,\b,\c);[a_1,a_2,b_1,b_2,c_1,c_2];([\a],[\v],[\c]))\}\]
that includes the double point $(B_0;(\a,\0,-\a);[a_1,a_2,0,0,-a_1,-a_2];([\a],[\v],[\a]))$:
\begin{equation}\label{family2}
\begin{minipage}{0.45\textwidth}
\begin{tikzpicture}[scale=.8]
\def\radius{1.3}
\def\radiustwo{1.838478}
\def\aa{1.125833}
\def\bb{.65}
\draw[thick] (0,0) circle (\radius cm);
\coordinate (target) at ({\radius*cos(120)},{\radius*sin(120)}); 
\coordinate (target2) at ({\radius*cos(-30)},{\radius*sin(-30)}); 
\foreach \angle in {-120, -90, ..., -60}{
\coordinate (vertex) at ({\radius*cos(\angle)},{\radius*sin(\angle)});
\coordinate (vertextwo) at 
({\aa+\radiustwo*(\radius*cos(\angle)-\aa)/
sqrt(2*(\radius)^2-2*\aa*\radius*cos(\angle)+2*\bb*\radius*sin(\angle))},
{-\bb+\radiustwo*(\radius*sin(\angle)+\bb)/
sqrt(2*(\radius)^2-2*\aa*\radius*cos(\angle)+2*\bb*\radius*sin(\angle))});
\fill (vertex) circle (1pt);
\draw (vertex) -- (target);
\draw (vertex) -- (target2);
\draw (vertextwo) -- (target2);}
\foreach \angle in {0, 30, ..., 60}{
\coordinate (vertex) at ({\radius*cos(\angle)},{\radius*sin(\angle)});
\coordinate (vertextwo) at 
({\aa+\radiustwo*(\radius*cos(\angle)-\aa)/
sqrt(2*(\radius)^2-2*\aa*\radius*cos(\angle)+2*\bb*\radius*sin(\angle))},
{-\bb+\radiustwo*(\radius*sin(\angle)+\bb)/
sqrt(2*(\radius)^2-2*\aa*\radius*cos(\angle)+2*\bb*\radius*sin(\angle))});
\fill (vertex) circle (1pt);
\fill (\aa,-\bb) circle (1pt);
\draw (\aa,-\bb) circle (2pt);
\draw (vertex) -- (target);
\draw (vertex) -- (target2);
\draw (vertextwo) -- (target2);}
\draw[thick] (target)--(target2);
\draw[thick]
({\aa+\radiustwo*(1/2)},{-\bb+\radiustwo*sqrt(3)/2})--
({\aa-.9*\radiustwo*(1/2)},{-\bb-.9*\radiustwo*sqrt(3)/2});
\draw
(\aa,-\bb)--(3,-\bb);
\draw[->,-latex] (\aa+.4,-\bb) arc (0:60:.4) node[right] {\footnotesize\;$\xi_\v$};
\fill (target) circle (1pt);
\filldraw[white] (target2) circle (2pt);
\draw (target2) circle (2pt);
\fill (target2) circle (1pt);
\node[below right] at (target2) {\footnotesize $C$};
\node[above left] at (target) {\footnotesize $B=B_0$};
\node[above right] at ({1.3*cos(60)},{1.3*sin(60)}) {$A(\tfrac\pi3)$};
\node[below left] at ({1.3*cos(-120)},{1.3*sin(-120)}) {$A(-\tfrac{2\pi}3)$};
\node at (0,-2.4) {\footnotesize Argument Calculation at a Double Point};
\end{tikzpicture}
\end{minipage}
\end{equation}
The value $[\v]$ is $\rho^{-1}(\xi_\v)$, where $\xi_\v$ is the argument of the tangent to the circle
at the vertex $C$. Since $\xi_\v$ is the limit of the arguments of the side vector $\b$ as it tends to $\0$ in the family,
it is the approach direction of the degeneration.
Thus the different values of $[\v]$ indeed parameterize approach directions at the double point,
as captured by the screens of \cite{FM94},
and they arise naturally in continuous families.

By inscribing $BC$ in circles of different radii we obtain different values for $\xi_\v$
at the double point $A=C$, reflecting the limits of different families of inscribed triangles with the 
common side-vector $\a=BC$, ranging over
all values in $\R/\pi$, and comprising all approach directions, see \eqref{appdir}.
\begin{equation}\label{appdir}
\begin{minipage}{0.3\textwidth}
\begin{tikzpicture}[scale=.8]
\foreach \h in {-1,0,1}
{\draw[domain={atan(\h)}:{180-atan(\h)},smooth,variable=\t] plot ({sqrt(1+(\h)^2)*cos(\t)},{-\h+sqrt(1+(\h)^2)*sin(\t)});
}
\draw[->,-latex] (1,0)--({1-(.5*sqrt(1+(-1)^2)*sin(atan(-1))))},{.5*sqrt(1+(-1)^2)*cos(atan(-1))}); 
\draw[->,-latex] (1,0)--({1-((.5*1.414)*sin(atan(0))))},{(.5*1.414)*cos(atan(0))}); 
\draw[->,-latex] (1,0)--({1-(.5*sqrt(1+(1)^2)*sin(atan(1))))},{.5*sqrt(1+(1)^2)*cos(atan(1))}); 
\draw[->,-latex] (1,0)--({1+(.5*1.414)*cos(atan(0))},{(.5*1.414)*sin(atan(0))}); 
\node[below] at (1.1,0) {\footnotesize $C=A$};
\node[left] at (-1.1,0) {\footnotesize $B$};
\draw[thick] (-1,0)--(1,0);
\fill (-1,0) circle (1pt);
\filldraw[white] (1,0) circle (2pt);
\draw (1,0) circle (2pt);
\fill (1,0) circle (1pt);
\node at (0,-.8) {\footnotesize Different Approach Directions};
\end{tikzpicture}
\end{minipage}
\end{equation}
Therefore double points are depicted as in \eqref{types}:
\[
\begin{tikzpicture}[scale=1.2] 
\node (A) at (1,0){};
\node (B) at (3,0) {};
\fill (A) circle (1pt);
\fill (B) circle (1pt);
\draw (B) circle (2pt);
\draw[->,-latex]([yshift=1.2]2.9,0)--([yshift=1.2]1.1,0);
\draw[->,-latex]([yshift=-1.2]1.1,0)--([yshift=-1.2]2.9,0);
\draw[->,-latex] (B)--({3+.5*cos(45)},{.5*sin(45)});
\draw[->,-latex] (B)--({3+.5*cos(135)},{.5*sin(135)});
\draw[->,-latex] (B)--({3+.5*cos(0)},{.5*sin(0)});
\draw[->,-latex] (B)--({3+.5*cos(90)},{.5*sin(90)});
\node[below] at (A) {\footnotesize $B$};
\node[below] at (B) {\footnotesize $C=A$};
\node at (2,-.8) {\footnotesize Double Points};
\end{tikzpicture}\]

Here we also have a reason for including approach directions
in the description of double points:
Since the interior angle $\alpha$ of every triangle
in the family \eqref{family2}
is constant $\pmod\pi$, and different for circles of different radii,
the same chord $BC$ with fixed data $(B_0;(\a,\0,-\a))$ appears in 
families with constant and distinct interior angles $\alpha$.
A model that incorporates interior angles will separate these degenerate classes.

\subsection{Incorporating Orientation}\label{orientation}
The necessity of including orientation
is discussed in \cite[p.46, 1.1.10]{Beh} and 
in \cite[p.305]{Montgomery}.
Essentially, a space that doesn't account for it fails to distinguish 
certain nonisomorphic families.
Such a space also fails to continuously track some continuous families:
Figure~\ref{family1} shows what clearly should qualify as a continuous family,
inscribed triangles with fixed chord $BC$
and variable vertex $A(\xi)=e^{\xi\i}$,
changing orientation at the degenerate double point $C=A$.
\begin{equation}\label{family1}
\begin{minipage}{0.4\textwidth}
\begin{tikzpicture}[scale=.8]
\def\radius{1.3}
\draw[thick] (0,0) circle (\radius cm);
\filldraw (0,0) circle (1pt);
\coordinate (target) at ({\radius*cos(120)},{\radius*sin(120)}); 
\coordinate (target2) at ({\radius*cos(-30)},{\radius*sin(-30)}); 
\foreach \angle in {-90, -70, ..., 30}{
\coordinate (vertex) at ({\radius*cos(\angle)},{\radius*sin(\angle)});
\fill (vertex) circle (1pt);
\draw (vertex) -- (target);
\draw (vertex) -- (target2);
\draw[thick] (target)--(target2);
}
\fill (target) circle (1.5pt);
\fill (target2) circle (1.5pt);
\node[below right] at (target2) {\footnotesize $C$};
\node[above left] at (target) {\footnotesize $B$};
\node[right] at ({\radius*cos(30)},{\radius*sin(30)}) {\footnotesize $A(\pi/3)$};
\node[below left] at ({\radius*cos(270)},{\radius*sin(270)}) {\footnotesize $A(-\pi/2)$};
\draw[->,-latex] (0,0)--(0,1.7);
\draw[->,-latex] (0,0)--(0,-1.7);
\draw[->,-latex] (0,0)--(1.7,0);
\draw[->,-latex] (0,0)--(-1.7,0);
\node at (0,-2.3) {\footnotesize Orientation Reversal in Families};
\end{tikzpicture}
\end{minipage}
\end{equation}
If orientation is not incorporated, the 
interior angle $\alpha$ undergoes a jump discontinuity at the double point.
For by Euclid the positive classes have some fixed $\alpha_+$,
the negative classes some fixed $\alpha_-$, 
and $|\alpha_+|+|\alpha_-|=\pi$, and the family is discontinuous at the double point when $\alpha_+\neq\pi/2$.
By including orientation and the accompanying sign convention
we compute $\alpha_+-\alpha_-=\pi$, hence $\alpha_+=\alpha_-\pmod\pi$, and the problem disappears:
At the double point, $(\alpha,\beta,\gamma)=(\alpha_+,0,-\alpha_+)=(\alpha_-,0,-\alpha_-)\pmod\pi$,
so the interior angles vary continuously.

\section{\Large Sphere and Torus}\label{ssss}

That $\S^2$ and $\T$ are fine moduli spaces of certain restricted 
triangle classes appears 
in \cite[1.1.10]{Beh} and \cite{BG23}, respectively.
Both are compactifications of the space of nondegenerate, labeled,
oriented triangle classes, but they handle
degenerating families differently, each collapsing the other's degenerates.
In this section we show how each is obtained from $\DD$ as a blow-down.

\subsection{The Sphere} 
Identifying $\DD$ with $\Bl_{D_\dbl}(\P(\X)_\R)$ via Theorem~\ref{main2} and $\P(\X)_\R$ with the 2-sphere $\S^2$ defines
a projection $\pi_{\S^2}:\DD\to\S^2$, and shows $\DD$ can be viewed as the blow-up of $\S^2$ at the image 
$\{\delta_\a,\delta_\b,\delta_\c\}$ of 
the three double points $D_\dbl=\{p_X,p_Y,p_Z\}$.
In this section we map out the classes in $\S^2$ via the projection from $\DD$, essentially
recovering the constructions in \cite{Kend84}, \cite[Section 1]{Kend85}, \cite{ES15}, 
\cite{Beh}, and \cite{Montgomery}.
As we will see, the degenerate points in
$\S^2$ are all collinear, and form the metric border 
between positively and negatively oriented nondegenerates. 

To map $\P(\X)_\R\isom\P^1(\C)_\R$ to $\S^2$, we map $\X=\V(X+Y+Z)\subset\P^2(\C)$ to 
$\V(Z)\isom\C^2$ with a special unitary transformation,
and then apply the Hopf map
from $\C^2-\{0\}$ to the unit sphere in $\R^3$ (compare \cite[p.312]{Kend84}, \cite[Exercise 1.39]{Beh},  and \cite[Theorem 1]{Montgomery}). 
Thus the classes of labeled, oriented triangles are given by the quotient
\[\begin{tikzpicture}
\node (A) at (-2,0) {$\DD$};
\node (B) at (0,0) {$\P(\X)$};
\node (C) at (2,0) {$\P(\V(Z))$};
\node (D) at (4,0) {$\P^1(\C)$};
\node (E) at (6,0) {$\S^2$};
\node[left] (AA) at (-1.5,-.8) {$([\a,\b,\c];(\alpha,\beta,\gamma))$};
\node (BB) at (0,-.8) {$[\a,\b,\c]$};
\node (CC) at (2,-.8) {$[\bs r,\bs s,\bs 0]$};
\node (DD) at (4,-.8) {$[\bs r,\bs s]$};
\node (EE) at (6,-.8) {$\mf{H}_{\bs i}({\bs r},{\bs s})$};
\draw[->,-latex] (A)--(B);
\draw[->,-latex] (B)--(C);
\draw[->,-latex] (C)--(D);
\draw[->,-latex] (D)--(E);
\draw[->,-latex] (AA)--(BB);
\draw[->,-latex] (BB)--(CC);
\draw[->,-latex] (CC)--(DD);
\draw[->,-latex] (DD)--(EE);
\draw[->, bend left=15,-latex] (A) to (E);
\node[above] at (5,-.1) {$\widesim$};
\node at (2,.8) {\footnotesize $\pi_{\S^2}$};
\end{tikzpicture}\]
For the right-multiplication matrix taking $\X$ to 
$\V(Z)$ we choose
\[\sfrac 16\,\sdbmx{3+\sqrt 3&-3+\sqrt 3&2\sqrt 3
\\-3+\sqrt 3&3+\sqrt 3&2\sqrt 3\\-2\sqrt 3&-2\sqrt 3&2\sqrt 3}\;\in\SU(3)\]
so that $[\bs r,\bs s]
=[\a+(2-\sqrt 3)\b,(2-\sqrt 3)\a+\b]$.
For the Hopf map
we identify $(\bs r,\bs s)\in\C^2-\{0\}$ with the quaternion $\bs r+\bs s\j\in\mbb H^\times$, and assign to it the point 
\[\mf{H}_{\i}(\bs r,\bs s):=(\bs r+\bs s\j)^{-1}\i(\bs r+\bs s\j)
=\sfrac{\left(\bs r\ol{\bs r}-\bs s\ol{\bs s},(\ol{\bs r}\bs s-\bs r\ol{\bs s})\bs i,\ol{\bs r}\bs s+\bs r\ol{\bs s}\right)}
{\bs r\ol{\bs r}+\bs s\ol{\bs s}}\]
The image is the unit sphere $\S^2$ in the space of pure imaginary quaternions
$\Im(\mbb H)\equiv\R^3$.
The stabilizer of $\i$ under right multiplication
is the subgroup $\C^\times\leq\mbb H^\times$, which shows
$\mf{H}_{\i}(\bs r,\bs s)=\mf{H}_{\i}(\bs r',\bs s')$ if and only if $[\bs r,\bs s]=[\bs r',\bs s']$.
Therefore $\mf{H}_{\i}$ gives the desired isomorphism $\P^1(\C)_\R\isom\S^2$.
The following images on $\S^2$ of different triangle types are now easy to verify.
\begin{Proposition}\label{sphere}
Let $X,Y,Z$ be the coordinate functions on $\R^3$, in which $\S^2$ is embedded as the unit sphere 
$\V(X^2+Y^2+Z^2-1)$. Then under the above projection $\pi_{\S^2}:\DD\to\S^2:$
\begin{enumerate}[\rm(a)]
\item
The degenerate classes $\DD_\dgen$ map to $\V(Y)\cap\S^2$ in $\R^3$,
such that distinct collinear points are separated, 
and the double-point divisors $E_{p_X},E_{p_Y},E_{p_Z}$ map to
$\delta_\a=(-\tfrac{\sqrt 3}2,0,\tfrac 12)$, $\delta_\b=(\tfrac{\sqrt 3}2,0,\tfrac 12)$, 
$\delta_\c=(0,0,-1)$, respectively.
\item
Positively oriented classes map to $Y<0$, negatively oriented classes to $Y>0$.
\item
The positively oriented equilateral point $([1,e^{2\pi\i/3},e^{4\pi\i/3}];(\frac\pi 3,\frac\pi 3,\frac\pi 3))$
maps to $[\bs r,\bs s]=[1,\i]\in\P^1(\C)$ and then to
$-\j$ on $\S^2$, 
and the negatively oriented counterpart to $[\bs r,\bs s]=[1,-\i]\in\P^1(\C)$ and then to $\j\in\S^2$.
\item
The isosceles classes map to three great circles at $X-\sqrt 3 Z=0$,
$X+\sqrt 3 Z=0$, and $X=0$, identified by odd sides $\a$, $\b$, and $\c$, respectively.
\item
The right classes map to three circles at 
$-\sqrt 3 X+Z=-1$, $\sqrt 3 X+Z=-1$, and $Z=\frac 12$, identified by hypotenuses
$\a,\b$, and $\c$, respectively.
\item
The obtuse classes map to three spherical caps of height $\frac 12$ and area $\pi$.
\end{enumerate}
\end{Proposition}

This agrees with the description in
\cite[2.6]{ES15}, see also \cite[Figure 2]{Montgomery}. 
Figure~\ref{spheres} shows side views of the parameterizing space $\S^2$ embedded in $\R^3$, using yellow and gray
shading for positively and negatively oriented classes, respectively.

\begin{equation}\label{spheres}
\begin{minipage}{0.4\textwidth}
\begin{tikzpicture}
\coordinate (A) at (-{sqrt(3)/2}, 1/2);
\coordinate (B) at ({sqrt(3)/2}, 1/2);
\coordinate (C) at (0, -1);
\coordinate (AA) at (-{sqrt(3)/2}, -1/2);
\coordinate (BB) at ({sqrt(3)/2}, -1/2);
\coordinate (CC) at (0, 1);
\fill (A) circle (1pt) node[left] {$\delta_\a$};
\fill (B) circle (1pt) node[right] {$\delta_\b$};
\fill (C) circle (1pt) node[below left] {$\delta_\c$};
\fill[yellow!50] (0,0) circle (1); 
\draw[thick,dashed] (A)--(B)--(C)--(A);
\draw[thick] (B)--(AA);
\draw[thick] (A)--(BB);
\draw[thick] (C)--(CC);
\draw[dashed,->,-latex](-1.5,0)--(1.5,0) node[right] {$X$};
\draw[dashed,->,-latex](0,-1.5)--(0,1.5) node[left] {$Z$};
\draw (0,0) arc (180:90:1 and 1) node[right] {equilateral};
\draw (1/2,{-sqrt(3)/2}) arc (180:90:.5 and -.5) node[right] {degenerate};
\draw (-1/4,1/2) arc (0:90:.5 and .5) node[left] {right};
\draw ({-sqrt(3)/3}, -1/3) arc (0:90:.5 and -.5) node[left] {isosceles};
\draw (0,-1/2) arc (0:73:1.5 and -.4); 
\draw[thick] (0,0) circle (1);
\filldraw[white] (0,0) circle (1pt);
\filldraw[white] (1/2,{-sqrt(3)/2});
\draw (0,0) circle (1.5pt);
\draw (1/2,{-sqrt(3)/2}) circle (1.5pt);
\draw (-1/4,1/2) circle (1.5pt);
\draw ({-sqrt(3)/3}, -1/3) circle (1.5pt);
\draw (0,-1/2) circle (1.5pt);
\node at (0,-2) {\footnotesize View of $\S^2$ from $-Y$-axis};
\end{tikzpicture}
\end{minipage}
\begin{minipage}{0.2\textwidth}
\begin{tikzpicture}
\coordinate (A) at (-{sqrt(3)/2}, 1/2);
\coordinate (B) at ({sqrt(3)/2}, 1/2);
\coordinate (C) at (0, -1);
\coordinate (AA) at (-{sqrt(3)/2}, -1/2);
\coordinate (BB) at ({sqrt(3)/2}, -1/2);
\coordinate (CC) at (0, 1);
\fill (A) circle (1pt) node[left] {$\delta_\b$};
\fill (B) circle (1pt) node[right] {$\delta_\a$};
\fill (C) circle (1pt) node[below left] {$\delta_\c$};
\fill[gray!50] (0,0) circle(1); 
\draw[thick,dashed] (A)--(B)--(C)--(A);
\draw[thick] (B)--(AA);
\draw[thick] (A)--(BB);
\draw[thick] (C)--(CC);
\draw[dashed,->,-latex](1.5,0)--(-1.5,0) node[left] {$X$};
\draw[dashed,->,-latex](0,-1.5)--(0,1.5) node[left] {$Z$};
\draw[thick] (0,0) circle (1);
\filldraw[white] (0,0) circle (1pt);
\draw (0,0) circle (1.5pt);
\node at (0,-2) {\footnotesize View from $+Y$-axis};
\end{tikzpicture}
\end{minipage}
\end{equation}


Figure~\ref{3dsphere} shows the 2-sphere with equilateral points on the $Y$-axis, 
degenerate great circle at $Y=0$,
double points on the great circle of degenerates, and the three isosceles great circles.

\tdplotsetmaincoords{80}{120}
\begin{equation}\label{3dsphere}
\begin{minipage}{0.2\textwidth}
\begin{tikzpicture}[tdplot_main_coords, scale=.8]
\begin{scope}
  \foreach \theta in {0,10,...,350} {
    \foreach \phi in {0,10,...,170} {

      \pgfmathsetmacro\xA{2*sin(\phi)*cos(\theta)}
      \pgfmathsetmacro\yA{2*sin(\phi)*sin(\theta)}
      \pgfmathsetmacro\zA{2*cos(\phi)}

      \pgfmathsetmacro\xB{2*sin(\phi+10)*cos(\theta)}
      \pgfmathsetmacro\yB{2*sin(\phi+10)*sin(\theta)}
      \pgfmathsetmacro\zB{2*cos(\phi+10)}

      \pgfmathsetmacro\xC{2*sin(\phi+10)*cos(\theta+10)}
      \pgfmathsetmacro\yC{2*sin(\phi+10)*sin(\theta+10)}
      \pgfmathsetmacro\zC{2*cos(\phi+10)}

      \pgfmathsetmacro\xD{2*sin(\phi)*cos(\theta+10)}
      \pgfmathsetmacro\yD{2*sin(\phi)*sin(\theta+10)}
      \pgfmathsetmacro\zD{2*cos(\phi)}

      \ifdim \yA pt > 0pt \else
      \ifdim \yB pt > 0pt \else
      \ifdim \yC pt > 0pt \else
      \ifdim \yD pt > 0pt \else
        \fill[yellow!50, opacity=0.5, draw=none]
          (\xA,\yA,\zA) --
          (\xB,\yB,\zB) --
          (\xC,\yC,\zC) --
          (\xD,\yD,\zD) -- cycle;
      \fi\fi\fi\fi
    }
  }
\end{scope}

\begin{scope}
  \foreach \theta in {0,10,...,350} {
    \foreach \phi in {0,10,...,170} {

      \pgfmathsetmacro\xA{2*sin(\phi)*cos(\theta)}
      \pgfmathsetmacro\yA{2*sin(\phi)*sin(\theta)}
      \pgfmathsetmacro\zA{2*cos(\phi)}

      \pgfmathsetmacro\xB{2*sin(\phi+10)*cos(\theta)}
      \pgfmathsetmacro\yB{2*sin(\phi+10)*sin(\theta)}
      \pgfmathsetmacro\zB{2*cos(\phi+10)}

      \pgfmathsetmacro\xC{2*sin(\phi+10)*cos(\theta+10)}
      \pgfmathsetmacro\yC{2*sin(\phi+10)*sin(\theta+10)}
      \pgfmathsetmacro\zC{2*cos(\phi+10)}

      \pgfmathsetmacro\xD{2*sin(\phi)*cos(\theta+10)}
      \pgfmathsetmacro\yD{2*sin(\phi)*sin(\theta+10)}
      \pgfmathsetmacro\zD{2*cos(\phi)}

      \ifdim \yA pt > 0pt \def\drawgray{1}
      \else\ifdim \yB pt > 0pt \def\drawgray{1}
      \else\ifdim \yC pt > 0pt \def\drawgray{1}
      \else\ifdim \yD pt > 0pt \def\drawgray{1}
      \else \def\drawgray{0}
      \fi\fi\fi\fi

      \ifnum\drawgray=1
        \fill[gray!50, opacity=0.5, draw=none]
          (\xA,\yA,\zA) --
          (\xB,\yB,\zB) --
          (\xC,\yC,\zC) --
          (\xD,\yD,\zD) -- cycle;
      \fi
    }
  }
\end{scope}
\draw[line width=0.5mm, domain=-90:90, smooth, variable=\t] plot ({2*cos(\t)},0,{2*sin(\t)}); 
\draw[line width=0.2mm, dashed, domain=90:270, smooth, variable=\t] plot ({2*cos(\t)},0,{2*sin(\t)}); 
\draw[thick, domain=-90:90, smooth, variable=\t] plot (0,{2*cos(\t)},{2*sin(\t)}); 
\draw[line width=0.2mm, dashed, domain=90:270, smooth, variable=\t] plot (0,{2*cos(\t)},{2*sin(\t)}); 
\draw[thick, domain=-60:120, smooth, variable=\t] plot ({sqrt(3)*cos(\t)},{2*sin(\t)},{cos(\t)}); 
\draw[line width=0.2mm, dashed, domain=0:360, smooth, variable=\t] plot ({sqrt(3)*cos(\t)},{2*sin(\t)},{cos(\t)}); 
\draw[thick, domain=-60:120, smooth, variable=\t] plot ({sqrt(3)*cos(\t)},{2*sin(\t)},{-cos(\t)}); 
\draw[line width=0.2mm, dashed, domain=0:360, smooth, variable=\t] plot ({sqrt(3)*cos(\t)},{2*sin(\t)},{-cos(\t)}); 
\draw[thick, domain=0:360, smooth, variable=\t] plot ({-1.01*cos(\t)},{sqrt(3)*1.01*cos(\t)},{2.02*sin(\t)}); 
\draw[thick,->,-latex] (0,0,0) -- (3.5,0,0) node[anchor=north west]{\footnotesize $X$};
\draw[thick,->,-latex] (0,0,0) -- (0,2.8,0) node[anchor=north west]{\footnotesize $Y$};
\draw[thick,->,-latex] (0,0,0) -- (0,0,2.7) node[anchor=east]{\footnotesize $Z$};
\node at (0,0,-3) {$\S^2\subset\R^3$};
\fill[white] (0,2,0) circle (2pt);
\draw (0,2,0) circle (2pt);
\fill (2,0,0) circle (2pt);
\fill (0,0,2) circle (2pt);
\fill ({sqrt(3)},0,1) circle (2pt) 
node[black,anchor=north west]{\footnotesize $\delta_\b$};
\fill ({-sqrt(3)},0,1) circle (2pt)
node[black,anchor=north east]{\footnotesize $\delta_\a$};
\fill (0,0,-2) circle (2pt)
node[black,anchor=north]{\footnotesize $\delta_\c$};

\end{tikzpicture}
\end{minipage}
\end{equation}

\subsection{The Torus}\label{aaas}  

Next we show how $\DD$ blows down to the $2$-torus $\T$.
First we review $\T$'s construction as a shape space, from \cite{BG23}.

\begin{Lemma}\label{rep}
Let $P,Q,R$ be coordinate functions on $\A^3(\R/\pi)$,
and put $\T=\V(P+Q+R)$.
The projection 
\begin{align*}
\pi_{\T }\,:\,\DD\;&\llr\A^3(\R/\pi)\\
([\a,\b,\c];(\alpha,\beta,\gamma))&\lm(\alpha,\beta,\gamma)
\end{align*}
has image $\T$. It sends positively oriented classes 
into the subset $P+Q+R=\pi\pmod{3\pi}$;
negatively oriented classes into $P+Q+R=-\pi\pmod{3\pi}$;
and degenerate classes $\a\b\c=\0$ to the boundary $PQR=0$, such that all distinct double points
of $\DD$ are separated.
\end{Lemma}

\begin{proof}
The space $\A^3(\R/\pi)$ is a 3-torus, and by \cite[Theorem 3.1]{BG23}
the subspace $\T$, which is the image of $\pi_\T$, is a 2-torus;
a compact abelian Lie group, and an $\R$-variety.
Figure~\ref{cube} depicts a fundamental domain, in which all angles are taken from
the interval $[0,\pi)$.
\begin{equation}\label{cube}
\tdplotsetmaincoords{60}{120}
\begin{minipage}{0.5\textwidth}
\begin{tikzpicture}[scale=2,tdplot_main_coords]
\coordinate (A) at (0,0,0);
\coordinate (B) at (1,0,0);
\coordinate (C) at (1,1,0);
\coordinate (D) at (0,1,0);
\coordinate (E) at (0,0,1);
\coordinate (F) at (0,1,1);
\coordinate (G) at (1,0,1);
\coordinate (H) at (1,1,1);
\draw (A)--(B)--(C)--(D)--cycle; 
\draw (A)--(E)--(F)--(D)--cycle; 
\draw (A)--(B)--(G)--(E)--cycle; 
\draw[->,-latex] (A)--(2,0,0) node[left] {$P$};
\draw[->,-latex] (A)--(0,2,0) node[right] {$Q$};
\draw[->,-latex] (A)--(0,0,1.5) node[left] {$R$};
\fill (A) circle (.5pt);
\fill (B) circle (.5pt) node[left] {\footnotesize $(\pi,0,0)$};
\fill (D) circle (.5pt) node[right] {\footnotesize $(0,\pi,0)$};
\fill (C) circle (.5pt);
\fill (E) circle (.5pt) node[left] {\footnotesize $(0,0,\pi)$};
\fill (F) circle (.5pt);
\fill (G) circle (.5pt);
\fill (H) circle (.5pt);
\fill[yellow, opacity=0.5] (1,0,0) -- (0,1,0) -- (0,0,1) -- cycle;
\draw[thick] (B)--(D)--(E)-- cycle;
\fill[gray, opacity=0.5] (1,1,0) -- (1,0,1) -- (0,1,1) -- cycle;
\draw[thick] (C)--(G)--(F)-- cycle;
\draw (G)--(H)--(F); 
\draw (H)--(C); 
\node at (1,1,-.5) {$\T\subset\A^3(\R/\pi)$};
\draw (.35,.85,1) arc (290:180:.5 and .1) node[right] {\footnotesize $P+Q+R=2\pi\pmod{3\pi}$};
\draw (.2,.9,.2) arc (270:180:.5 and .1) node[right] {\footnotesize $P+Q+R=\pi\pmod{3\pi}$};
\end{tikzpicture}
\end{minipage}
\end{equation}
The boundaries of the triangular regions, where one angle is
$0\pmod\pi$, are identified
on opposite faces of the cube, and all 8 vertices, where all
angles are $0\pmod\pi$, are identified to a single point.

By \cite[Theorem 4.1]{BG23}, the positively oriented classes are in bijective
correspondence with points of the yellow region
$P+Q+R=\pi\pmod{3\pi}$, negatively oriented classes with the gray region
$P+Q+R=2\pi\pmod{3\pi}$, and neutrally oriented inscribable classes with the boundary $\V(PQR)=\V(P)\cup\V(Q)\cup\V(R)$.
Angles in $[0,\pi)$ adding to $2\pi$ may be replaced by angles in $(-\pi,0]$ adding to $-\pi$.
All collinear points map to the single vertex $(0,0,0)\in\A^3(\R/\pi)$, and the double points
bijectively to the interior boundary points of $\V(P),\V(Q)$, and $\V(R)$.
\end{proof}

Here are the two orientations of nondegenerate classes,
with the gluing instructions
indicated on the edges, which are the exceptional fibers of the blow-up of $\S^2$ at
$\delta_\a,\delta_\b,\delta_\c$, as indicated.

\begin{equation}\label{angles}
\begin{minipage}{.4\textwidth}
\centering
\begin{tikzpicture}[scale=1.5]
\coordinate (A) at ({-sqrt(3)/2},0);
\coordinate (B) at (0,3/2);
\coordinate (C) at ({sqrt(3)/2},0);
\coordinate (D) at (0,1/2); 
\coordinate (E) at (0,0);
\coordinate (F) at ({-sqrt(3)/4},3/4); 
\coordinate (G) at ({sqrt(3)/4},3/4);
\fill (A) circle (.5pt);
\fill (B) circle (.5pt);
\fill (C) circle (.5pt);
\fill (D) circle (.5pt);
\fill (E) circle (.5pt);
\fill (F) circle (.5pt);
\fill (G) circle (.5pt);
\fill[yellow!50] (G)--(E)--(F)--(G) (A)--(B)--(C)--(A);
\draw[thick,dashed] (G)--(E)--(F)--(G);
\draw[thick] (A)--(B)--(C)--(A);
\draw[thick] (G)--(A);
\draw[thick] (B)--(E);
\draw[thick] (F)--(C);
\filldraw[white] (D) circle (1pt); 
\draw (D) circle (1pt);
\draw[thick,->] (C)--(-1/4,0);
\draw[thick,->>] (A)--({-sqrt(3)/6},1);
\draw[thick,->>>] (B)--({3*sqrt(3)/8},3/8);
\draw (D) arc (180:90:1 and .75) node[right] {equilateral};
\draw (1/2,0) arc (180:90:.5 and -.25) node[right] {degenerate};
\draw ({(sqrt(3)+1)/4},{(3-sqrt(3))/4}) arc (90:0:.75 and .4);
\draw (-1/6,3/4) arc (0:90:.5 and .5) node[left] {right};
\draw ({-sqrt(3)/4},1/4) arc (0:90:.5 and .25) node[left] {isosceles};
\draw (1/2,0) circle (1.5pt);
\draw (-1/6,3/4) circle (1.5pt);
\draw ({-sqrt(3)/4},1/4) circle (1.5pt);
\draw ({(sqrt(3)+1)/4},{(3-sqrt(3))/4}) circle (1.5pt);
\node[left] at (F) {$E_{\delta_\b}$};
\node[right] at (G) {$E_{\delta_\a}$};
\node[below] at (E) {$E_{\delta_\c}$};
\node at (0,-5/8)  {\footnotesize Positive Orientation};
\end{tikzpicture}
\end{minipage}
\begin{minipage}{.4\textwidth}
\centering
\begin{tikzpicture}[scale=1.5]
\coordinate (A) at ({-sqrt(3)/2},0);
\coordinate (B) at (0,-3/2);
\coordinate (C) at ({sqrt(3)/2},0);
\coordinate (D) at (0,-1/2);
\coordinate (E) at (0,0);
\coordinate (F) at ({-sqrt(3)/4},-3/4);
\coordinate (G) at ({sqrt(3)/4},-3/4);
\fill (A) circle (.5pt);
\fill (C) circle (.5pt);
\fill (D) circle (.5pt);
\fill (E) circle (.5pt);
\fill (B) circle (.5pt);
\fill (F) circle (.5pt);
\fill (G) circle (.5pt);
\fill[gray!50] (G)--(E)--(F)--(G) (A)--(B)--(C)--(A); 
\draw[thick,dashed] (G)--(E)--(F)--(G);
\draw[thick] (A)--(B)--(C)--(A);
\draw[thick] (G)--(A);
\draw[thick] (B)--(E);
\draw[thick] (F)--(C);
\filldraw[white] (D) circle (1pt); 
\draw (D) circle (1pt);
\draw[thick,->>>] (A)--({-sqrt(3)/8},-9/8);
\draw[thick,->>] (B)--({sqrt(3)/2-sqrt(3)/6},-3/2+1);
\draw[thick,->] (C)--(-1/4,0);
\node[left] at (F) {$E_{\delta_\a}$};
\node[right] at (G) {$E_{\delta_\b}$};
\node[above] at (E) {$E_{\delta_\c}$};
\node at (0,-2)  {\footnotesize Negative Orientation};
\end{tikzpicture}
\end{minipage}
\end{equation}
See \cite[Figure 2]{ES15} for an identical depiction of the triangle on the left.

\begin{Proposition}
The projection
$\pi_\T:\DD\to\T$ of Lemma~\ref{rep} is the blow-up of $\T$ at $0$.
\end{Proposition}

\begin{proof}
The blow-up of any smooth real 2-dimensional manifold (such as $\T$) over a point
has fiber isomorphic to $\P^1(\R)$, and the coordinates on the fiber parameterize 
the approach directions on the tangent plane.
By e.g. \cite[IV-16]{EH}, to prove the proposition
it is enough to show:
the fiber $E_0:=\pi_\T^{-1}(0)$  
is isomorphic to $\P^1(\R)$;
$\pi_\T|_{\DD-E_0}$ is an isomorphism;
and points on $E_0$ parameterize approaches to $0$ on $\T$,
i.e., if $L(t):L(0)=0$ is a curve on $\T$ then the limit on the fiber is 
\[\lim_{t\to 0}\pi^{-1}_\T(L(t))=([L'(0)];0)\in\P^1(\R)\times\T\]

Since their interior angles are $0\pmod\pi$,
the triples of side-vectors in $E_0$ are collinear,
given by $E_0=\{([\a,\b,\c];(0,0,0)):\a,\b,\c\in\R,\a+\b+\c=\0\}$. Therefore $E_0\isom\P^1(\R)$.

Each $(\alpha,\beta,\gamma)\neq 0$
determines $[\a,\b,\c]$ by the rule
\[[\a,\b,\c]
=[1-e^{-2\alpha\i},-1+e^{2\beta\i},e^{-2\alpha\i}-e^{2\beta\i}]\]
using the description of vertices 
$(A,B,C)=(e^{2\beta\i},e^{-2\alpha\i},1)$ inscribed in the unit circle.
Since this evidently inverts $\pi_\T$ on $\T-\{0\}$, 
$\pi_\T|_{\DD-E_0}$ is an isomorphism.


Let $(\alpha_0,\beta_0,\gamma_0)\neq 0$ be on $\T$, and set $L(t)=t(\alpha_0,\beta_0,\gamma_0)$,
a line in $\T$ with constant direction $[\alpha_0,\beta_0,\gamma_0]$.
The class of side-vectors over $L(t)$ is
\[[\a(t),\b(t),\c(t)]:=[1-e^{-2t\alpha_0\i},-1+e^{2t\beta_0\i},e^{-2t\alpha_0\i}-e^{2t\beta_0\i}]\]
Taking the limit and applying l'H\^opital's rule yields
\[\lim_{t\to 0}([\a(t),\b(t),\c(t)];t(\alpha_0,\beta_0,\gamma_0))=
([\alpha_0,\beta_0,\gamma_0];(0,0,0))\in E_0\]
Therefore the points on $E_0$ parameterize approaches to $0$, as desired.
We conclude $\DD=\Bl_0(\T)$.
\end{proof}

\section{\Large Quotient Stack}\label{mss}

The symmetric group
$S_3$ acts on $\TT$ by fixing $B_0$ and permuting
the letter labels on each
$(B_0;(\bs x,\bs y,\bs z);[a_1,a_2,b_1,b_2,c_1,c_2];([\u],[\v],[\w]))$,
and $\Z_2$ acts by reversing orientation, which attaches a minus sign to 
$\bs x,\bs y,\bs z$.
These actions commute, and since $S_3\times\Z_2\isom D_6$ we have a group action \[D_6\times\TT\isim\TT\]
This result appears in \cite[1.1.1]{Beh}, especially Exercise 1.38.
There is an obvious induced action 
\[D_6\times\DD\isim\DD\]
The points of the quotient $\DD/D_6$
correspond to absolute (unlabeled, unoriented) similarity classes,
and we obtain an algebraic description of this set as the quotient stack $[\DD/D_6]$ as follows.

By \cite[Section 4, Example 4.8]{DM69} or \cite[p.88]{Beh}, the stack quotient $[\DD/D_6]$ 
is the category whose objects $(E/B,f)$ are diagrams
\[\begin{tikzcd}[ampersand replacement=\&]
E\&\DD\\ B
\arrow[r,"f",from=1-1,to=1-2,-latex]
\arrow[d,from=1-1,to=2-1,-latex]
\end{tikzcd}\]
where $B\in\Var{\R}$, $E/B$ is a $D_6$-torsor, and $f$ is $D_6$-equivariant;
and whose morphisms 
\[(g,\phi):(E'/B',f')\lr(E/B,f)\] are given by morphisms $g:B'\to B$ in $\Var{\R}$
and $D_6$-equivariant isomorphisms $\phi:E'\to E\times_B B'\to E$ satisfying $f'=f\circ\phi$.
We write $[\DD/D_6](B)$ for the subcategory of $[\DD/D_6]$ determined by fixed base $B$.
A morphism $g:B'\to B$ in $\Var{\R}$ determines the (pullback) functor 
\[g^*:[\DD/D_6](B)\lr[\DD/D_6](B')\]
which is base extension, taking $(E/B,f)$ to $(E\times_B B',f\circ p_1)$, 
where $p_1:E\times_B B'\to E$ is the projection.
By construction
each $[\DD/D_6](B)$ is a groupoid (see \cite[Section 1.1.1]{Beh}), and since $\DD$ is a scheme its functor of points is a sheaf,
hence $[\DD/D_6]$ is an algebraic stack, which is a sheaf fibered in groupoids.
Thus we obtain:

\begin{Theorem}
$[\DD/D_6]$ 
is a smooth stack of (unrestricted) similarity classes of triangles.
\end{Theorem}

\newpage

\bibliographystyle{alpha} 
\bibliography{../../MathDocs/MasterBib.bib}

\end{document}